\documentclass[english]{article}

\usepackage[utf8]{inputenc}
\usepackage[T1]{fontenc}

\usepackage{geometry}
\usepackage{babel}
\usepackage{float}
\usepackage{mathtools}
\usepackage{bm}
\usepackage{bbm}
\usepackage{amsmath,amssymb,amsthm}
\usepackage{dsfont}
\usepackage{graphicx}
\usepackage{natbib}
\usepackage{enumitem}
\usepackage{authblk}
\usepackage{color}
\usepackage{booktabs}

\usepackage{qsw}

\usepackage[
  colorlinks,
  linkcolor=blue,
  citecolor=blue,
  urlcolor=blue,
  breaklinks
]{hyperref}
\usepackage{cleveref}

\makeatletter

\usepackage{algorithm}
\usepackage{algpseudocode}
\theoremstyle{plain}
\newtheorem{theorem}{Theorem}[section]
\newtheorem{lemma}{Lemma}[section]
\newtheorem{corollary}{Corollary}[section]

\newtheorem{definition}{Definition}[section]

\newtheorem{proposition}{Proposition}[section]

\theoremstyle{definition}
\newtheorem{remark}{Remark}[section]

\newtheorem{example}{Example}[section]

\newcommand{\hinv}{h_+^{\leftarrow}}

\newcommand{\Toff}[2]{T(#1,#2)}

\newcommand{\hupup}{h_+\mathbin{\circ}h_+}                 
\newcommand{\hinvinv}{h_+^{\leftarrow}\mathbin{\circ}h_+^{\leftarrow}}  

\newcommand{\mgn}[1]{\tau+\sqrt{\tau\,V(#1)}}

\newcommand{\errsum}[1]{\errone(\cM_{\smlr};#1)+\errtwo(\cM_{\smlr};#1)}

\newif\ifshowcomments
\showcommentstrue      

\definecolor{cm}{RGB}{0,0,200}
\definecolor{ks}{RGB}{160,32,240}

\ifshowcomments
  \newcommand{\cm}[1]{\textcolor{cm}{[CM: #1]}}
  \newcommand{\ks}[1]{\textcolor{ks}{[KS: #1]}}
\else
  \newcommand{\cm}[1]{}
  \newcommand{\ks}[1]{}
\fi

\showcommentsfalse

\DeclareMathOperator{\HyperGeom}{HyperGeom}
\DeclareMathOperator{\Binom}{Binom}

\newcommand{\fnull}{f_0}
\newcommand{\falt}{f_1}

\newcommand{\convclosure}{\overline{\mathrm{conv}}}
\newcommand{\smlr}{\mathcal{S}_{\mathrm{MLR}}}     
\newcommand{\etaND}{\eta_{n,\delta}}               
\newcommand{\tauND}{\tau_{n,\delta}}               
\newcommand{\hup}{h_+}                             
\newcommand{\hdown}{h_-}                           
\newcommand{\fstar}{f^\star}                       
\newcommand{\dTV}{d_{\mathrm{TV}}}                 

\@ifundefined{showcaptionsetup}{}{%
  \PassOptionsToPackage{caption=false}{subfig}}
\usepackage{subfig}

\makeatother

\begin{document}

\allowdisplaybreaks

\title{When Are Trade-Off Functions Testable from Finite Samples?}

\author{Kaining Shi$^{*}$}
\author{Qiaosen Wang$^{*}$}
\author{Cong Ma}
\affil{Department of Statistics, University of Chicago}

\date{\today}

\maketitle

\begingroup
\renewcommand{\thefootnote}{\fnsymbol{footnote}}
\setcounter{footnote}{1}
\footnotetext{These two authors contributed equally.}
\addtocounter{footnote}{-1}
\endgroup

\begin{abstract}
We study finite-sample inference for the trade-off function of two unknown
probability distributions, the function that traces the optimal type~I/type~II
error frontier in binary testing. 
Given samples from distributions \(P\) and
\(Q\), we consider the problem of testing whether their trade-off function lies above a
benchmark curve \(f_0\) or falls below a weaker benchmark \(f_1\). 
Without
structural restrictions, this problem is impossible uniformly over nonparametric classes. We identify a sharp condition under which it becomes
possible. The key structural assumption is that the Neyman--Pearson rejection
regions for \((P,Q)\) are attainable, up to null sets, by a prescribed class
\(\mathcal S\) of measurable sets. Within this exact attainability framework,
finite Vapnik--Chervonenkis dimension of \(\mathcal S\) is both sufficient and
necessary for nontrivial finite-sample testing. We construct a test with
nonasymptotic error guarantees: type~I error control is valid without assuming
attainability, while power holds uniformly over attainable alternatives
satisfying an explicit separation condition.  
By inverting the test, we also obtain simultaneous confidence
bands for the whole trade-off curve. Finally, we study the sharpness and
robustness of the procedure. In the monotone likelihood-ratio model, we derive
local separation rates and prove matching lower bounds up to logarithmic
factors. We also allow approximate, rather than exact, attainability; this
extension yields finite-sample guarantees for univariate log-concave
distributions by approximating their rejection regions with unions of
intervals.
\end{abstract}

\section{Introduction}\label{sec:intro}

How well can two probability distributions be distinguished from finite samples?
This question arises across several areas of modern statistics and machine learning. In privacy auditing~\cite{Auditing3, Auditing2}, one asks whether the output distributions of a randomized mechanism on neighboring datasets are sufficiently similar to satisfy a target privacy guarantee. In tolerant distribution testing~\cite{price2021tolerant}, one asks whether two unknown distributions are close enough under one threshold or far enough apart under another. Although these problems are often studied separately, they share a common statistical object: the \emph{trade-off function}.

For two distributions $P$ and $Q$ on a common measurable space, the trade-off function $\Toff{P}{Q}$, introduced by \cite{GDP}, maps each type~I error level to the minimum achievable type~II error when testing $P$ against $Q$. It therefore
gives the full Neyman--Pearson error frontier. Lower bounds on this function
encode indistinguishability guarantees. For example, \(f\)-differential
privacy~\cite{GDP} can be expressed as a uniform lower bound on the trade-off
functions of neighboring output distributions, while total variation distance
corresponds to a particular piecewise-linear benchmark curve. This makes the trade-off function a natural common language for a broad class of inference problems.

In this paper, we study the following composite hypothesis testing problem
\[
H_0:\; \Toff{P}{Q}\succeq \fnull,
\qquad\text{versus}\qquad
H_1:\; \Toff{P}{Q}\not\succeq \falt,
\]
based on independent samples $X_{1:n}\sim P$ and $Y_{1:n}\sim Q$, where $\fnull\succeq \falt$ are benchmark trade-off functions  and
\(f\succeq g\) means \(f(\alpha)\geqslant g(\alpha)\) for every
\(\alpha\in[0,1]\). Under the null, $P$ and $Q$ are at least as hard to distinguish as prescribed by $\fnull$; under the alternative, the trade-off curve falls below the weaker benchmark $\falt$ somewhere. This formulation includes, as special cases, testing privacy guarantees
against a benchmark profile and tolerant testing under total variation.

The basic question is whether such hypotheses can be tested from finite samples. In full generality, the answer is negative. The trade-off function is an infinite-dimensional object, and without structural assumptions on the pair $(P,Q)$, no finite-sample procedure can estimate it accurately enough to support nontrivial uniform testing. At the other extreme, if $(P,Q)$ belongs to a known parametric family, then the problem reduces to classical finite-dimensional inference. The interesting regime lies between these extremes: one seeks structural assumptions that are weak enough to cover broad nonparametric families, but strong enough to make the trade-off curve statistically identifiable.

Our starting point is the Neyman--Pearson lemma. Writing \(\mu=(P+Q)/2\), \(p=dP/d\mu\), and \(q=dQ/d\mu\), the trade-off function is generated by the likelihood-ratio level sets
\[
\{x:q(x)>\lambda p(x)\}, \qquad \lambda \in [0, \infty].
\]
These are the rejection regions that trace out the optimal type~I/type~II frontier. Thus, finite-sample inference for the trade-off curve is governed by the geometry of these witness sets. If they belong to a structured family whose probabilities can be estimated uniformly, then the trade-off curve itself becomes testable.

To formalize this idea, we introduce \(\mathcal S\)-attainability as a modeling language for classes of distribution pairs. Given a set family \(\mathcal S\), a pair \((P,Q)\) is \(\mathcal S\)-attainable if every Neyman--Pearson likelihood-ratio level set can be represented, up to null sets, by a member of \(\mathcal S\). The class \(\mathcal S\) describes the geometry of the most
informative tests between \(P\) and \(Q\). For example, in a monotone
likelihood-ratio model on the real line, the relevant level sets are
half-lines.

This formulation separates the statistical problem into two components. The
first is a structural component: violations of the benchmark must be witnessed
by sets in \(\mathcal S\). The second is a complexity component: probabilities
of sets in \(\mathcal S\) must be estimable uniformly from samples. Our main
result shows that, within the exact \(\mathcal S\)-attainability framework,
finite Vapnik--Chervonenkis (VC) dimension is the threshold for finite-sample
testability. If \(\mathrm{VC}(\mathcal S)<\infty\), we construct a
nonasymptotic test with uniform guarantees. If
\(\mathrm{VC}(\mathcal S)=\infty\), no nontrivial finite-sample test exists,
even when the hypotheses are separated by fixed benchmark curves.

The proposed test scans over candidate rejection regions \(S\in\mathcal S\)
and compares the empirical pair \((P_n(S),Q_n(S^c))\) with the benchmark curve
after applying simultaneous confidence corrections. Its validity has an
important asymmetry. Type~I error control is assumption-free: it holds for
arbitrary distribution pairs \((P,Q)\), even if the pair is not
\(\mathcal S\)-attainable. The attainability assumption is needed only for
power, because it ensures that any population-level violation of the benchmark
has a witness inside the search class. This feature is useful in applications
such as privacy auditing, where false discoveries may be more costly than loss
of power under model misspecification.
By inverting the test, we also obtain simultaneous confidence bands for the whole trade-off
curve. 

Finally, we study the sharpness and robustness of the procedure.
We derive
local separation rates showing how the detectable gap between \(f_0\) and
\(f_1\) depends on the local geometry of the benchmark curve. In the monotone
likelihood-ratio model, we prove matching lower bounds up to logarithmic
factors, showing that these local rates are essentially sharp. We also 
relax exact attainability to an approximate version. This gives a
bias--variance trade-off: enlarging \(\mathcal S\) improves approximation of
the optimal rejection regions but increases the uniform estimation error. As an
application, we show that univariate log-concave pairs can be handled by
approximating their likelihood-ratio level sets with unions of intervals.

\paragraph{Paper organization.}
The rest of the paper develops this program in stages. \Cref{sec:formulation} sets up the formal framework and establishes the impossibility of testing without structural assumptions. \Cref{sec:mlr} develops the monotone likelihood ratio case as a warm-up. \Cref{sec:general} introduces \(\mathcal S\)-attainability, proves the necessity of finite VC dimension, presents the general test, and inverts it into simultaneous confidence bands. \Cref{sec:rates} extracts local detection boundaries and establishes matching lower bounds in the MLR model. \Cref{sec:misspecification} extends the theory to approximate attainability and applies it to log-concave distributions on \(\mathbb R\). \Cref{sec:computation} derives polynomial-time algorithms via the dual formulation. \Cref{sec:experiments} reports numerical experiments. \Cref{sec:related-work,sec:discussion} conclude with related work and open questions.

\section{Problem formulation}\label{sec:formulation}

In this section, we set up the formal framework of the paper. We first define the trade-off function and then state the trade-off testing
problem, together with two canonical examples. We then show that, without
further assumptions on the distribution pair, this testing problem is impossible
to solve. 

\subsection{The trade-off function}

Fix a measurable space $(\mathcal{X}, \mathcal{A})$. 
Let $P$ and $Q$ be two probability distributions on this measurable space.
We have the following definition of the trade-off function.

\begin{definition}[Trade-off function, {\cite{GDP}}]\label{def:tradeoff}

The \emph{trade-off function} $\Toff{P}{Q}: [0,1] \to [0,1]$ is defined by
\[
\Toff{P}{Q}(\alpha) \;\coloneqq\; \inf_{\varphi:\, E_P[\varphi] \leqslant \alpha} E_Q[1-\varphi],
\]
where the infimum is over all measurable (randomized) tests $\varphi$.
\end{definition}

\noindent The trade-off function characterizes the optimal type~I/type~II
error frontier in binary hypothesis testing between $P$ and $Q$. 
It interpolates between two extremes:
$\Toff{P}{Q}(\alpha) \equiv 1-\alpha$ when $P=Q$, and $\Toff{P}{Q}(\alpha) \equiv 0$ when $P$ and $Q$ have
disjoint support.

A function $f:[0,1]\to[0,1]$ is a valid trade-off
function if and only if it is convex, continuous, non-increasing, and
satisfies $f(\alpha)\leqslant 1-\alpha$ for all $\alpha\in[0,1]$; see
Proposition~2.2 of \cite{GDP}.

\subsection{The trade-off testing problem}

Given independent samples $X_{1:n} \sim P$ and $Y_{1:n} \sim Q$, we study the
composite hypothesis test
\begin{equation}\label{eq:H}
H_0: \Toff{P}{Q} \succeq \fnull, \qquad \text{versus} \qquad H_1: \Toff{P}{Q} \not\succeq \falt,
\end{equation}
where $\fnull$ and $\falt$ are two benchmark trade-off functions\footnote{Throughout, any trade-off benchmark \(f\) is extended to \([1,\infty)\) by
setting \(f(\alpha)=0\) for \(\alpha>1\).}
satisfying $\fnull \succeq \falt$. Here $f \succeq g$ means
$f(\alpha)\geqslant g(\alpha)$ for all $\alpha\in[0,1]$, and
$f \not\succeq g$ means $f(\alpha_0)<g(\alpha_0)$ for some $\alpha_0 \in [0,1]$.

Under $H_0$, the pair $(P,Q)$ is at least as hard
to distinguish as prescribed by $\fnull$; under $H_1$,  it is strictly  more distinguishable than allowed
by the weaker benchmark $\falt$ at some level $\alpha$.

Let $\cM$ be a set of distribution pairs under consideration. 
For a test $\psi = \psi(X_{1:n}, Y_{1:n})$, we can define the type~I and type~II errors by
\begin{align*}
\errone(\cM;\, \psi) &\coloneqq \sup_{(P,Q) \in \cM:\, \Toff{P}{Q} \succeq \fnull} (P^n \times Q^n)(\psi), \\
\errtwo(\cM;\, \psi) &\coloneqq \sup_{(P,Q) \in \cM:\, \Toff{P}{Q} \not\succeq \falt} (P^n \times Q^n)(1 - \psi).
\end{align*}
When $\cM$ is unrestricted, we omit it and write $\errone(\psi)$ and
$\errtwo(\psi)$ instead. 
We seek tests that simultaneously control $\errone$ and $\errtwo$ uniformly over $\cM$.

\begin{figure}[t]
  \centering
  \begin{minipage}[t]{0.45\textwidth}
    \centering
    \includegraphics[width=\linewidth]{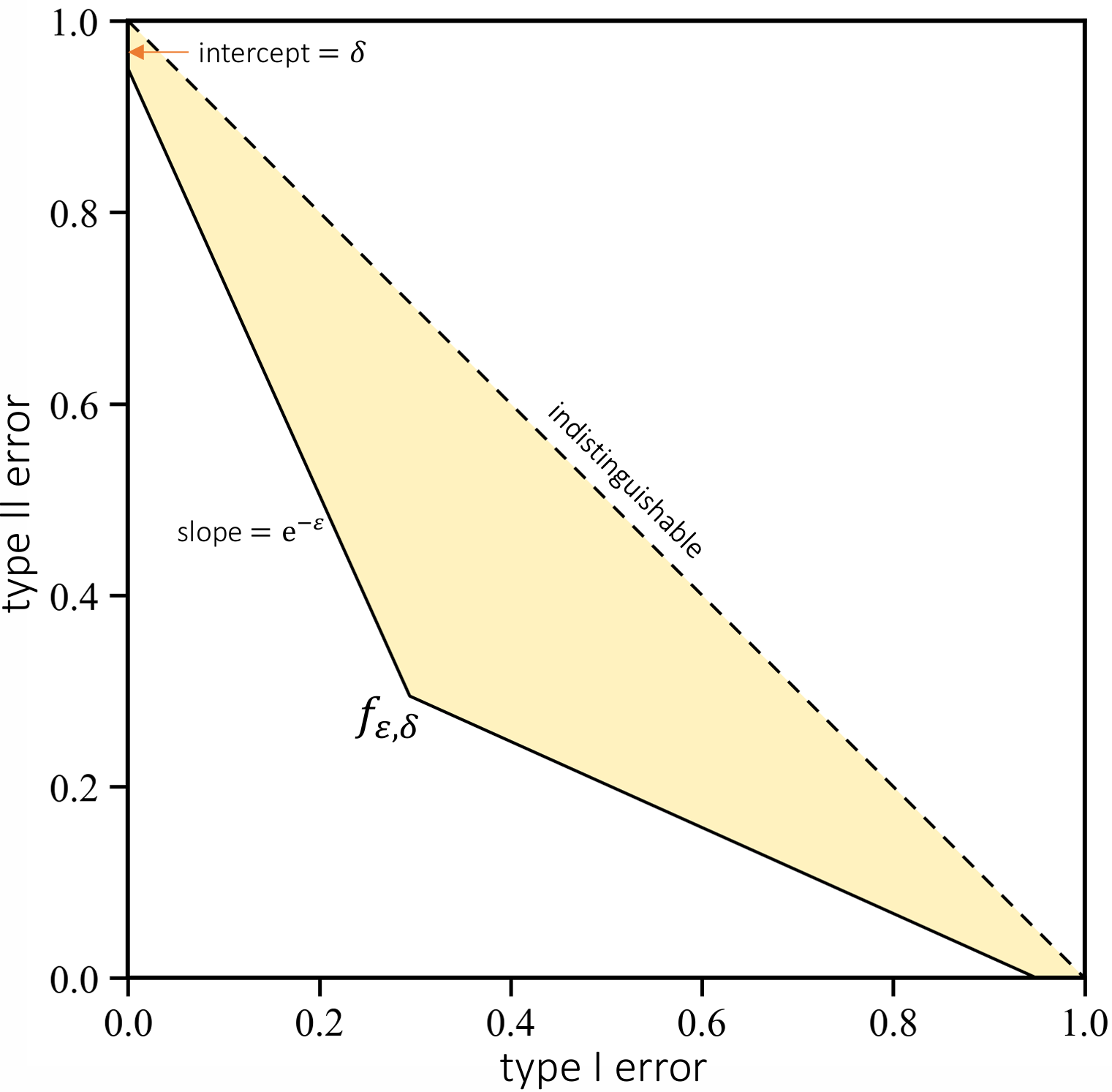}
    \caption{Illustration of testing $(\varepsilon,\delta)$-DP}
    \label{fig:testing-dp}
  \end{minipage}
  \hfill
  \begin{minipage}[t]{0.45\textwidth}
    \centering
    \includegraphics[width=\linewidth]{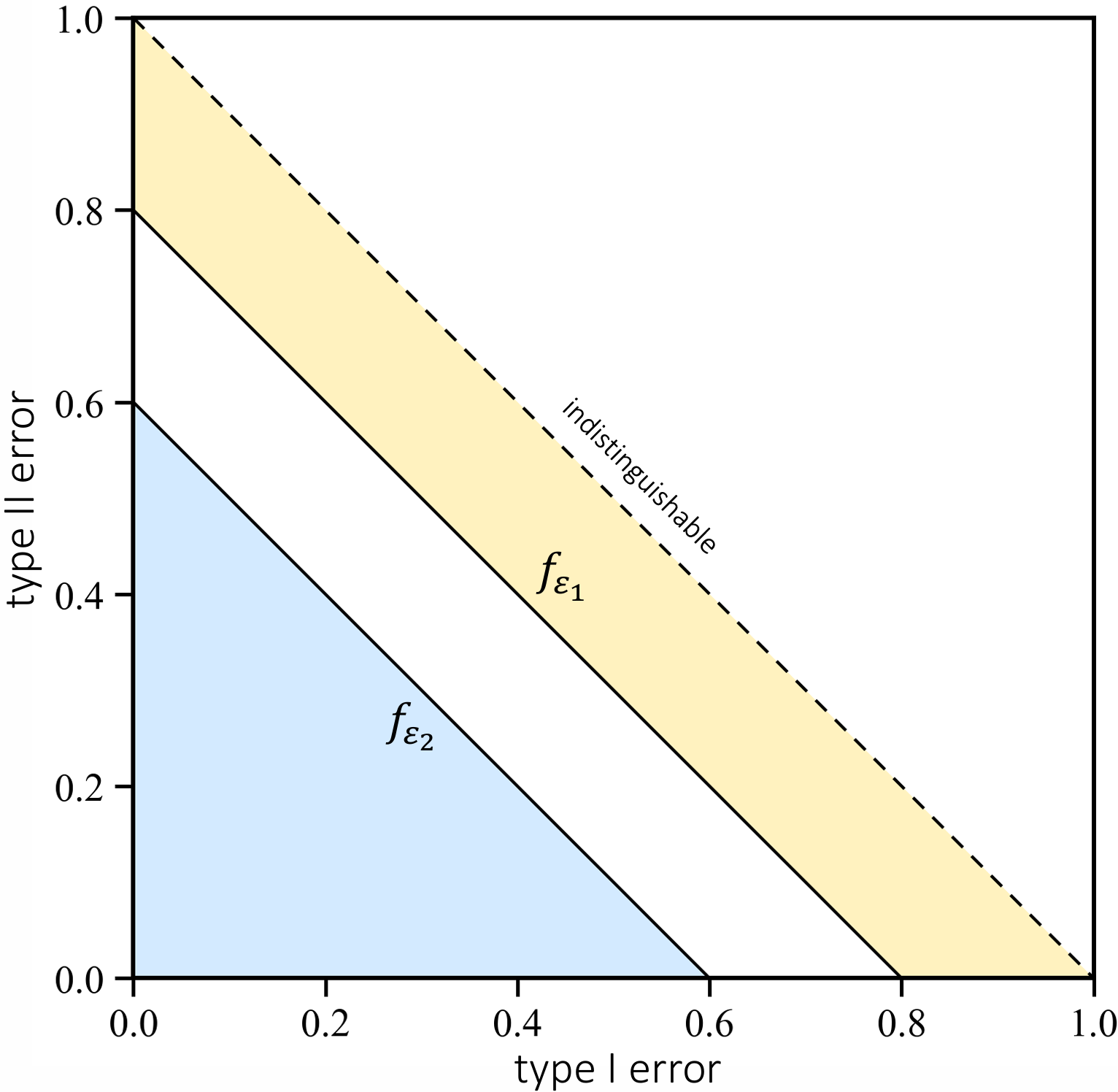}
    \caption{Illustration of tolerant distribution testing}
    \label{fig:tolerant-testing}
  \end{minipage}
\end{figure}

\subsection{Two examples}

We briefly record two canonical examples that fit into the trade-off testing
problem~\eqref{eq:H}.

\paragraph{Example 1: auditing differential privacy.}
In the literature on differential privacy, we say a randomized mechanism $\mathcal A$ is $f$-differentially private~\cite{GDP}
if for every neighboring dataset pair $(D,D')$,
\[
\Toff{\mathcal A(D)}{\mathcal A(D')}\succeq f.
\]
The classical $(\varepsilon,\delta)$-differential privacy guarantee~\cite{dwork2006our} is the
special case
\[
f_{\varepsilon,\delta}(\alpha)
=
\max\Bigl\{
0,\,
1-\delta-e^\varepsilon\alpha,\,
e^{-\varepsilon}(1-\delta-\alpha)
\Bigr\}.
\]
See \Cref{fig:testing-dp} for the graph of $f_{\varepsilon,\delta}$.
Accordingly, fixing a neighboring dataset pair $(D, D')$, testing whether a mechanism satisfies one privacy benchmark or
violates a weaker one is an instance of \eqref{eq:H}.

\paragraph{Example 2: tolerant distribution testing.}

The central problem of tolerant distribution testing~\cite{price2021tolerant} is
\begin{align}\label{eq:tolerant-testing}
H_0:\ d_{\mathrm{TV}}(P,Q)\leqslant \varepsilon_1, \qquad \text{versus} \qquad
H_1:\ d_{\mathrm{TV}}(P,Q)>\varepsilon_2.
\end{align}
This is also a trade-off testing problem in disguise.
To see this, for $\varepsilon\in[0,1]$, define the trade-off function
\[
g_\varepsilon(\alpha)=(1-\varepsilon-\alpha)_+.
\]
See \Cref{fig:tolerant-testing} for an illustration. A standard
characterization of total variation gives
\[
d_{\mathrm{TV}}(P,Q)\leqslant \varepsilon
\quad\Longleftrightarrow\quad
\Toff{P}{Q}\succeq g_\varepsilon.
\]
Hence the tolerant testing problem is exactly the trade-off testing problem
with benchmarks
\[
\fnull=g_{\varepsilon_1},
\qquad
\falt=g_{\varepsilon_2}.
\]
This shows that our formulation extends the usual ``close versus far''
paradigm from total variation to general benchmark curves.

\subsection{Structural conditions are necessary}

If $(P,Q)$ belongs to a known parametric family, such as the Gaussian location family
\[
\cM = \{(N(\theta_P, 1), N(\theta_Q, 1)) : \theta_P, \theta_Q \in \mathbb{R}\},
\]
then the trade-off function is determined by finitely many parameters, and the
testing problem~\eqref{eq:H} reduces to comparing them. At the opposite
extreme, if no structure is imposed, we show that the problem is impossible:

\begin{corollary}\label{cor:impossibility}
Let $\cM$ be the collection of all pairs of distributions on $\mathbb{R}$.
Then for any trade-off functions $\fnull$ and $\falt$ with $\falt(0) > 0$, any
$n \geqslant 1$, and any test $\psi$,
\[
\errone(\cM;\, \psi) + \errtwo(\cM;\, \psi) \;\geqslant\; 1.
\]
\end{corollary}

\begin{remark}
  This impossibility result is later subsumed by our general impossibility result (\Cref{thm:vc-necessary}). In fact, even if we restrict $\cM$ to be all pairs of continuous distributions on $\mathbb{R}$, testing is still impossible.  
\end{remark} 

This result shows that, over the unrestricted nonparametric model, finite
samples cannot yield a uniformly nontrivial test of the benchmark hypotheses.
Thus some structural restriction on $(P,Q)$ is unavoidable.

\section{A motivating example: the monotone likelihood ratio case}\label{sec:mlr}
Before developing the general framework, we consider the monotone likelihood-ratio (MLR) setting. 
This case illustrates the central mechanism of the paper: the Neyman--Pearson rejection regions have simple geometry and their probabilities can be uniformly estimated from finite samples. 

\subsection{The MLR condition}

We first recall the monotone likelihood ratio assumption.

\begin{definition}[Monotone likelihood ratio]\label{def:mlr}
A pair of continuous distributions $(P,Q)$ on $\mathbb{R}$ satisfies the
\emph{monotone likelihood ratio} (MLR) condition if the likelihood ratio
$x \mapsto dQ/dP(x)$ is monotone.
\end{definition}

\noindent Let $\mathcal M_{\mathrm{MLR}}$ denote the collection of all
continuous pairs $(P,Q)$ on $\mathbb R$ satisfying the MLR condition.
\medskip

This assumption is classical and includes many important families of distributions. Examples
include Gaussian distributions with common variance, Laplace distributions with
common scale, and, more generally, regular exponential families of the form
$g(t\mid\theta) = h(t)\,c(\theta)\,e^{w(\theta)t}$ whenever $w(\theta)$ is
monotone.

The key consequence of MLR is immediate from the Neyman--Pearson lemma: for
every threshold $\lambda \geqslant 0$, the likelihood-ratio level set
\[
\{x : dQ/dP(x) > \lambda\}
\]
is a half-line. Hence every optimal rejection region for testing $P$ against
$Q$ belongs to the class
\[
\smlr = \{(-\infty, a),\; (a, +\infty) : a \in \mathbb{R} \cup \{\pm\infty\}\}.
\]
This observation has two important implications. First, the trade-off function
$\Toff{P}{Q}$ is fully determined by the probabilities of sets in
$\mathcal{S}_{\mathrm{MLR}}$. Second, because this class is very simple, these
probabilities can be simultaneously estimated from a finite sample (cf.~the
Dvoretzky--Kiefer--Wolfowitz inequality below).

\subsection{A test based on the DKW inequality}

Let $P_n = \frac{1}{n}\sum_{i=1}^n \delta_{X_i}$ and $Q_n = \frac{1}{n}\sum_{i=1}^n \delta_{Y_i}$
denote the empirical distributions. The following lemma is a direct
consequence of the classical Dvoretzky--Kiefer--Wolfowitz inequality.

\begin{lemma}[DKW inequality]\label{lemma:DKW}
For any $\delta\in(0,1)$, with probability at least $1-\delta$,
\[
\sup_{S\in\smlr}\big|P_n(S)-P(S)\big|
\;\leqslant\;
\sqrt{\frac{1}{2n}\log\frac{2}{\delta}}.
\]
\end{lemma}

\noindent Define
\[
\etaND = \sqrt{\tfrac{1}{2n}\log\tfrac{4}{\delta}}.
\]
Motivated by \Cref{lemma:DKW}, we consider the test
\begin{equation}\label{eq:mlr-test}
\psi_{n,\delta}(\smlr)
\;=\;
\indi\Big\{\exists\, S \in \smlr:\;
Q_n(S^{\complement}) + \etaND
\;<\;
\fnull\big(P_n(S) + \etaND\big)\Big\}.
\end{equation}
In words, the test scans over all threshold sets and rejects whenever it finds
one for which the empirically estimated type~II error, even after the
confidence correction, is smaller than the null benchmark $\fnull$ allows at the
corresponding corrected type~I level.

The following theorem shows that this simple procedure already overcomes the
impossibility result from \Cref{sec:formulation} once the distribution pair is
restricted to the MLR class.

\begin{theorem}\label{thm:mlr}
Fix $\delta \in (0,1)$. The test $\psi_{n,\delta}$ defined in
\eqref{eq:mlr-test} satisfies:

\begin{enumerate}[label=(\roman*)]
\item \textbf{Type~I error control (assumption-free).}
\[
\errone(\psi_{n,\delta})
=
\sup_{(P,Q):\, \Toff{P}{Q}\succeq \fnull}
(P^n \times Q^n)(\psi_{n,\delta})
\;\leqslant\;
\delta.
\]

\item \textbf{Type~II error control (under MLR).}
If $\fnull$ and $\falt$ satisfy the separation condition
\begin{equation}\label{eq:mlr-sep}
\falt(\alpha)
\;\leqslant\;
\max\big\{\fnull(\alpha + 2\etaND) - 2\etaND,\; 0\big\}
\qquad \text{for all } \alpha \in [0,1],
\end{equation}
then
\[
\errtwo(\mathcal M_{\mathrm{MLR}};\, \psi_{n,\delta})
=
\sup_{\substack{(P,Q) \in \mathcal M_{\mathrm{MLR}}:\\ \Toff{P}{Q}\not\succeq \falt}}
(P^n \times Q^n)(1 - \psi_{n,\delta})
\;\leqslant\;
\delta.
\]
\end{enumerate}
\end{theorem}

\noindent See \Cref{sec:proof-mlr} for the proof.
\medskip

A notable feature of \Cref{thm:mlr} is the asymmetry between the two
guarantees. Type~I error control does not rely on the MLR assumption at all:
it holds for arbitrary distribution pairs $(P,Q)$. The structural assumption enters only in the power analysis: under MLR, any
violation of the benchmark $\falt$ has a Neyman--Pearson witness in
$\mathcal{S}_{\mathrm{MLR}}$.
This asymmetry is especially appealing in applications such as privacy
auditing: even if the structural model is misspecified, the procedure still
avoids false alarms; what may be lost is power, not validity.

\section{General framework}\label{sec:general}

We now move from the MLR warm-up to the general theory. The previous section
isolated the two ingredients that make finite-sample trade-off testing
possible: a structural condition that reduces the trade-off function to
probabilities of sets in a prescribed class, and a complexity condition that
guarantees these probabilities can be estimated uniformly from data. In this
section, we formalize both ingredients and show that together they yield a
sharp possibility theory for testing trade-off functions.

\subsection{$\mathcal{S}$-attainability}

In the MLR setting, the key structural fact was that every Neyman--Pearson
level set belongs to $\smlr$. We now abstract this into a general definition.

\begin{definition}[$\mathcal S$-attainability]\label{def:attainability}
Given a set class $\mathcal S \subset \mathcal A$ containing $\emptyset$ and $\mathcal X$, a pair
$(P,Q)$ is $\mathcal S$-attainable if, writing $\mu=(P+Q)/2$,
$p=dP/d\mu$, and $q=dQ/d\mu$, for every $\lambda \in [0,\infty]$ there exists
$S\in\mathcal S$ such that
\[
\mathbf 1_S = \mathbf 1_{\{q>\lambda p\}}
\qquad
\mu\text{-a.e.}
\]
Here, when $\lambda=\infty$, $\mathbf 1_{\{q>\lambda p\}}=\mathbf 1_{\{p=0\}}$.
Denote by $\cM_{\mathcal S}$ the collection of all pairs that are
$\mathcal S$-attainable.
\end{definition}

\noindent In words, $(P,Q)$ is $\mathcal S$-attainable if every level set of
the likelihood ratio can be represented, up to a null set, by a member of
$\mathcal S$. For continuous distributions on \(\mathbb{R}\), the MLR
condition implies \(\smlr\)-attainability, i.e.,
\[
\mathcal M_{\mathrm{MLR}} \subseteq \mathcal M_{\smlr}.
\]
Thus, the continuous MLR model from \Cref{sec:mlr} is a special case of
the general \(\mathcal S\)-attainability framework.

The importance of attainability is that it reduces the trade-off function to a
geometric object built only from probabilities of sets in \(\mathcal S\).

\begin{proposition}\label{prop:reduction}
If $(P,Q) \in \cM_{\mathcal S}$, then for all $\alpha \in [0,1]$,
\[
\Toff{P}{Q}(\alpha)
=
\inf\Big\{\beta : (\alpha, \beta) \in \convclosure\big\{(P(S),\, Q(S^{\complement})) : S \in \mathcal S\big\}\Big\}.
\]
Here, $\convclosure(C)$ denotes the closed convex hull of a set $C$.
\end{proposition}

\noindent See \Cref{sec:proof-convex-closure} for the proof.
\medskip

\Cref{prop:reduction} completes the reduction step: inferring
$\Toff{P}{Q}$ is equivalent to estimating the collection of probabilities
\[
\{P(S),\, Q(S^{\complement}) : S \in \mathcal S\}.
\]
In finite samples these quantities are observed with noise, so recovering
$\Toff{P}{Q}$ requires uniform control of the estimation error over all
$S \in \mathcal S$. This is where the complexity of the class \(\mathcal S\)
enters.

\subsection{VC dimension as the complexity measure}

The relevant notion of complexity is the Vapnik--Chervonenkis
dimension~\cite{vapnik2015uniform}.

\begin{definition}[VC dimension]\label{def:vc}
Let $\mathcal X$ be a set and $\mathcal S$ a collection of subsets of
$\mathcal X$. A subset $A \subseteq \mathcal X$ is \emph{shattered} by
$\mathcal S$ if
\[
\mathrm{card}\{A \cap S : S \in \mathcal S\}
=
2^{\mathrm{card}(A)}.
\]
The \emph{VC dimension} of $\mathcal S$, denoted $\VC(\mathcal S)$, is the
largest cardinality of a subset that can be shattered by $\mathcal S$.
\end{definition}

Finite VC dimension is the classical threshold for uniform convergence of
empirical probabilities over a set class. In our setting, it is also the
threshold for nontrivial testing within the attainability framework. We first
show the negative direction.

\begin{theorem}\label{thm:vc-necessary}
Let $\cS$ be a collection of measurable sets on $\mathbb{R}^d$ or $\bbZ$ with
$\VC(\cS) = +\infty$. Then for any $n \geqslant 1$, any trade-off functions
$\fnull$ and $\falt$ with $\falt(0) > 0$, and
any measurable test $\psi$,
\[
\errone(\cM_\cS;\, \psi) + \errtwo(\cM_\cS;\, \psi) \;\geqslant\; 1.
\]
\end{theorem}

\noindent See \Cref{sec:VC-necessity} for the proof. 
\medskip

The obstruction is a finite-sample ``needle-in-a-haystack'' phenomenon. If
\(\mathcal S\) has infinite VC dimension, then it can shatter arbitrarily large
finite sets. On such a shattered set, we take the null distribution to be
uniform and construct alternatives by tilting the likelihood ratio on an
unknown subset \(A\). Since \(A\) is realized by some \(S_A\in\mathcal S\),
each alternative has a witness set that certifies a violation of the benchmark.
But when the shattered set is much larger than the sample size, the subset
\(A\) is effectively hidden: averaging over all choices of \(A\) gives a
distribution of the data that is nearly indistinguishable from the null. Hence
no finite-sample test can uniformly find the relevant witness, even though such
a witness exists for every alternative.
\begin{remark}
We emphasize that the impossibility result for $\mathrm{VC}(\mathcal S)=\infty$ concerns testing over a
fixed \emph{exact} attainability class $\mathcal M_{\mathcal S}$. It does not preclude testing broader
families whose likelihood-ratio level sets are only approximable by finite-VC classes; this is precisely
the role of the approximate-attainability framework developed in \Cref{sec:misspecification}.
\end{remark}

\subsection{Testing by normalized uniform convergence}\label{sec:test_by_normalized_vc}

We now turn to the positive direction, assuming that 
$\operatorname{VC}(\mathcal S)<\infty$. Then empirical probabilities over
$\mathcal S$ satisfy a uniform concentration bound. For our purposes, the
appropriate form is a variance-adaptive, or normalized, uniform convergence
inequality.

\begin{lemma}[\cite{vapnik2015uniform}, see \cite{dasgupta2007general} Lemma 1]\label{lem:nuc}
Let $\mathcal S$ be a class of sets with
$\operatorname{VC}(\mathcal S)<\infty$, and let\footnote{In this paper, we always assume that $n$ is large enough such that \(\tauND \leqslant \tfrac{1}{2}\).}
\[
\tauND=\frac{4\midb{\VC(\cS)\log(2n)+\log(32/\delta)}}{n}.
\]
Then with probability at least $1-\frac{\delta}{4}$,
\begin{align*}
-\min\litb{\sqrt{\tauND P(S)},\ \tauND+\sqrt{\tauND P_n(S)}}
\leqslant P_n(S)-P(S)
\leqslant \min\litb{\sqrt{\tauND P_n(S)},\ \tauND+\sqrt{\tauND P(S)}}
\end{align*}
holds simultaneously for all $S\in\cS$.
\end{lemma}

Unlike the DKW inequality, which yields a uniform error of order \(\etaND\)
regardless of the underlying probability, \Cref{lem:nuc} adapts to the
local variance of the Bernoulli random variable. In particular, the confidence
width scales like \(\sqrt{\tau_{n,\delta}P(S)}\) when \(P(S)\) is small. This
sharper normalization is what allows the general procedure to recover the
correct local separation rates in~\Cref{sec:rates}.

With the help of normalized uniform convergence, we are ready to introduce the test. Define two helper functions $V: [0,\infty) \to [0,1]$ and $h: [0,\infty) \to [0,1]$: 
\[
V(t)=\min\{t, (1-t)_+\},
\qquad \text{and} \qquad
\hup(t)= \min \{ t+\sqrt{\tau_{n,\delta}V(t)}+\tau_{n,\delta},1\}.
\]
We can then define the test
\begin{align}\label{eq:general-test}
\psi_{n,\delta}(\cS;X_{1:n},Y_{1:n})
=
\indi\bigb{\exists\, S\in\cS:\ \hup(Q_n(S^\complement))< \fnull\litb{\hup(P_n(S))}}.
\end{align}
This is the direct analogue of the MLR test from
\Cref{sec:mlr}. The procedure searches over the candidate witness sets in
\(\mathcal S\) and rejects if one of them certifies, after simultaneous
confidence correction, that the benchmark \(\fnull\) cannot hold.

The next theorem gives the main finite-sample guarantee for this test.

\begin{theorem}\label{thm:general}
Fix $\delta\in(0,1)$ and a set collection $\cS$ with $\VC(\cS)<+\infty$.
The test $\psi_{n,\delta}$ defined in \eqref{eq:general-test} satisfies the
following.

\begin{enumerate}[label=(\roman*)]
\item \textbf{Type~I error control (assumption-free).}
\[
\errone(\psi_{n,\delta})
=
\sup_{(P,Q):\ \Toff{P}{Q}\succeq \fnull}
(P^n\times Q^n)(\psi_{n,\delta})
\;\leqslant\;
\delta.
\]

\item \textbf{Type~II error control (under $\cS$-attainability).}
If $\fnull$ and $\falt$ satisfy, for all $\alpha \in [0,1]$,
\begin{align} \label{eq:general-sep}
\falt(\alpha)\leqslant \hinvinv \circ \fnull\circ \hupup(\alpha),
\end{align}
then
\[
\errtwo(\cM_\cS;\, \psi_{n,\delta})
=
\sup_{(P,Q)\in\cM_\cS:\ \Toff{P}{Q}\nsucceq \falt}
(P^n\times Q^n)(1-\psi_{n,\delta})
\;\leqslant\;
\delta.
\]
Here $\hinv(t)=\inf\{s\in[0,1]:h_+(s)\geqslant t\}$ 
is the left-continuous generalized inverse of \(\hup\).
\end{enumerate}
\end{theorem}

\noindent See \Cref{sec:proof-general} for the proof.
\medskip

\Cref{thm:general} is the central positive result of this section. It
exhibits the same asymmetry that already appeared in the MLR case, but now in
full generality. Type~I error control uses only uniform concentration over
\(\mathcal S\). As a result, it is entirely assumption-free: the test remains
valid for arbitrary pairs \((P,Q)\), even when the model is misspecified and
the pair is not \(\mathcal S\)-attainable.

Type~II error control, by contrast, requires attainability. This is
unavoidable. If a true violation of the benchmark cannot be witnessed by a set
in \(\mathcal S\), then no search over \(\mathcal S\) can be guaranteed to
detect it. Thus, the reduction and complexity steps play genuinely different
roles: attainability identifies the right search space, while finite VC
dimension makes that search space estimable.

\subsection{Confidence bands for the trade-off curve}

The testing procedure from \Cref{thm:general} can be inverted to obtain
simultaneous confidence bands for the entire unknown trade-off function. As in
the testing problem, the key idea is to combine uniform concentration over the
class \(\cS\) with the representation of \(\Toff{P}{Q}\) in terms of the pairs
\((P(S),Q(S^\complement))\).

Define the lower confidence map
\[
h_-(t)=t-\sqrt{\tauND\,V(t)}-\tauND,
\]
which is parallel to the upper confidence map $\hup$.
Using these bounds, define for each \(\alpha\in[0,1]\)
\begin{align*}
\hat L(\alpha)
&\coloneqq
\Bigl[
\inf_{S\in\cS:\, h_-(P_n(S))\leqslant \alpha}
h_-\bigl(Q_n(S^\complement)\bigr)
\Bigr]\vee 0,\\
\hat U(\alpha)
&\coloneqq
\Bigl[
\inf_{S\in\cS:\, h_+(P_n(S))\leqslant \alpha}
h_+\bigl(Q_n(S^\complement)\bigr)
\Bigr]\wedge (1-\alpha).
\end{align*}
Let \(\hat L_{\mathrm{GCM}}\) denote the greatest convex minorant of \(\hat L\).
The next corollary shows that these envelopes provide a simultaneous
confidence band for the whole curve.

\begin{corollary}\label{cor:confidence-bands}
Fix \(\delta\in(0,1)\) and assume \(\operatorname{VC}(\mathcal S)<\infty\).
Then, with probability at least \(1-\delta\),
\[
\Toff{P}{Q}(\alpha) \leqslant \hat U(\alpha)
\qquad
\text{for all }\alpha\in[0,1],
\]
for arbitrary distribution pairs \((P,Q)\). Moreover, if in addition
\((P,Q)\in\mathcal M_{\mathcal S}\), then
\[
\hat L_{\mathrm{GCM}}(\alpha)
\leqslant
\Toff{P}{Q}(\alpha)
\leqslant
\hat U(\alpha)
\qquad
\text{for all }\alpha\in[0,1].
\]
\end{corollary}

\noindent The role of the convex minorant is natural: a trade-off function is
always convex, so projecting the raw lower envelope \(\hat L\) onto the class
of convex functions gives the sharpest lower band compatible with the shape
constraint. See \Cref{sec:proof-confidence-bands} for the proof.
\medskip

Thus the same asymmetry as in \Cref{thm:general} appears at the level of
confidence bands. The upper band is fully assumption-free: it is valid for
arbitrary \((P,Q)\) as long as \(\VC(\cS)<\infty\). The lower band, by
contrast, requires \(\cS\)-attainability, because only under that structural
condition can the trade-off function be recovered from the probabilities of
sets in \(\cS\). In this sense, the confidence-band result is the direct
inversion of the testing theorem.

\subsection{Adaptive testing over nested VC classes}

The procedure in \Cref{thm:general} requires knowledge of the exact VC
class \(\cS\). In practice, however, one is often given a nested sequence of
classes
\begin{align*}
\cS_1\subseteq\cS_2\subseteq \cS_3\subseteq\cdots
\end{align*}
with strictly increasing VC dimensions. For example, in \(\mathbb R\), the class of
unions of at most \(k\) intervals forms such a sequence as \(k\) increases; see~\Cref{eq:union-interval}.

Given a pair \((P,Q)\) known to be \(\cS_K\)-attainable for some
\(K\in\mathbb N\), one may run the testing procedure from
\Cref{thm:general} with \(\cS_K\). However, if \((P,Q)\) is in fact
\(\cS_k\)-attainable for some much smaller \(k\ll K\), then the procedure
based on \(\cS_K\) can be suboptimal in its detection boundary. This raises a
natural question: given an unknown pair \((P,Q)\) lying in a nested sequence of
VC classes \(\cS_1\subseteq\cS_2\subseteq\cdots\), can we adapt to the smallest
class \(\cS_k\) that makes \((P,Q)\) \(\cS_k\)-attainable? The following
result gives an affirmative answer.

\begin{corollary}\label{cor:adaptive}
Fix $\delta\in(0,1)$ and a nested sequence
$\cS_1\subseteq\cS_2\subseteq\cdots$ of strictly increasing VC dimension. Let
$\psi^{(k)}_{n,6\delta/(\pi^2k^2)}$ be the testing procedure in
\Cref{thm:general} specialized to \(\cS_k\) with tolerance level
\(6\delta/(\pi^2k^2)\). Consider the test
\begin{align*}
\tilde{\psi}_{n,\delta}=\max_{k\in \bbN}\psi^{(k)}_{n,6\delta/(\pi^2k^2)}.
\end{align*}
Then we have:
\begin{enumerate}[label=(\roman*)]
\item \textbf{Type~I error control (assumption-free).}
\[
\errone(\tilde{\psi}_{n,\delta})
=
\sup_{(P,Q):\ \Toff{P}{Q}\succeq \fnull}
(P^n\times Q^n)(\tilde{\psi}_{n,\delta})
\;\leqslant\;
\delta.
\]

\item \textbf{Adaptive type~II error control (under $\cS_k$-attainability).}
For each \(k\in \bbN\), if \(\fnull\) and \(\falt\) satisfy for all
\(\alpha \in [0,1]\)
\begin{align*}
\falt(\alpha)\leqslant \hinvinv \circ \fnull\circ \hupup(\alpha),
\end{align*}
where \(\hup\) and \(\hinv\) are defined as in \Cref{thm:general} with
\(\tauND\) replaced by
\begin{align*}
\tilde{\tau}^{(k)}_{n,\delta}
=
\frac{4\midb{\VC(\cS_k)\log(2n)+\log(64k^2/\delta)}}{n},
\end{align*}
then
\[
\errtwo(\cM_{\cS_k};\, \tilde{\psi}_{n,\delta})
=
\sup_{(P,Q)\in\cM_{\cS_k}:\ \Toff{P}{Q}\nsucceq \falt}
(P^n\times Q^n)(1-\tilde{\psi}_{n,\delta})
\;\leqslant\;
\delta.
\]
\end{enumerate}
\end{corollary}

Comparing \Cref{cor:adaptive} with
\Cref{thm:general}, we see that the cost of adaptation appears through
the replacement of \(\log(32/\delta)\) by \(\log(64k^2/\delta)\). Since the
sequence \(\{\cS_k\}\) is assumed to have increasing VC dimension, one must
have \(\VC(\cS_k)\geqslant k\) for all \(k\in\mathbb N\). Therefore,
\[
\log(64k^2/\delta)-\log(32/\delta)=O(\log(\VC(\cS_k))),
\]
which can be absorbed into the leading term \(\VC(\cS_k)\log(2n)\). As a
result, the adaptive procedure incurs only a negligible order-wise inflation in
the detection boundary.

\section{Optimal local separation rates}\label{sec:rates}

\Cref{thm:general} gives a finite-sample test and an exact separation
condition for its power. That condition, however, is expressed through the
confidence map \(h_+\) and its generalized inverse, and therefore does not
immediately reveal the statistical difficulty of the problem. In this section
we translate the condition into an interpretable local detection boundary and
ask whether that boundary is tight.

The resulting boundary has a simple interpretation. A violation of the
benchmark at level \(\alpha\) can be detected only up to the uncertainty in two
coordinates: the type~I coordinate \(\alpha\), and the corresponding type~II
coordinate \(\fnull(\alpha)\). We show that this two-sided local uncertainty is
not merely an artifact of our proof; in the half-interval model it is also
necessary, up to logarithmic factors.

The following local modulus captures these two sources of uncertainty. For a
trade-off function \(f\) and \(r\geqslant 0\), define
\begin{align}
\Delta_f(\alpha;r)
\coloneqq 
f(\alpha)-f\!\bigl(\alpha+r+\sqrt{rV(\alpha)}\bigr)
+
\Bigl(\sqrt{r\,V(f(\alpha))}\Bigr)\wedge f(\alpha),
\label{eq:local-modulus}
\end{align}
where, as before, \(V(t)=\min\{t,(1-t)_+\}\), and we extend \(f\) by setting
\(f(\alpha)=0\) for \(\alpha>1\).

The first term is the input-side uncertainty: it measures how much the benchmark
can change when the type~I level is shifted by the local empirical error
\(r+\sqrt{rV(\alpha)}\). The second term is the output-side uncertainty in
estimating the type~II coordinate \(f(\alpha)\). Thus
\(\Delta_f(\alpha;r)\) is the local amount by which the alternative benchmark
must fall below \(f\) in order for a violation at level \(\alpha\) to be
statistically visible at scale~\(r\).

\subsection{Sufficient local separation}

The first result shows that the abstract separation condition in
\eqref{eq:general-sep} is implied by a simple local gap condition. In words, it
is enough for the benchmark gap \(\fnull(\alpha)-\falt(\alpha)\) to dominate the
local modulus of \(\fnull\) at the estimation scale \(\tau_{n,\delta}\).

\begin{proposition}[Sufficient local separation]
\label{prop:upper-rate}
There exists a universal constant \(C>0\) such that
\eqref{eq:general-sep} is implied by
\begin{equation}
\fnull(\alpha)-\falt(\alpha)
\;\geqslant\;
C\,\Delta_{\fnull}(\alpha;\tau_{n,\delta})
\qquad\forall \alpha\in[0,1].
\label{eq:upper-rate-clean}
\end{equation}
\end{proposition}

\noindent See \Cref{sec:proof-separation} for the proof.
\medskip

The interpretation is simple: the benchmark gap must dominate the local
statistical uncertainty at the effective estimation scale
\(\tau_{n,\delta}\). The dependence on the witness class \(\cS\) enters only
through this scale. To understand the implication of the condition, it remains
to evaluate \(\Delta_{\fnull}(\alpha;\tau_{n,\delta})\) for concrete benchmark
curves. We do this after establishing that the same local modulus is also
necessary, up to logarithmic factors, for the half-interval model.

\subsection{Necessity in the half-interval model}

We next ask whether the modulus in \Cref{prop:upper-rate} reflects
a real statistical barrier. 
The answer is affirmative already for the half-interval class
\(\mathcal S_{\mathrm{MLR}}\), which has constant VC dimension.
Thus the local form of the upper bound is not an
artifact of the proof.

\begin{theorem}[Necessary local separation in the half-interval model]
\label{thm:new-lower-bound}
For every \(\delta\in(0,1/2)\), there exists a constant \(C_\delta>0\) such
that if for some \(\alpha\in[0,1]\),
\begin{equation}
\fnull(\alpha)-\falt(\alpha)
<
C_\delta\,\Delta_{\fnull}(\alpha;1/n),
\label{eq:lower-rate-clean}
\end{equation}
then
$
\inf_{\psi}
\Bigl\{
\errsum{\psi}
\Bigr\}
\geqslant 2\delta.
$
\end{theorem}

\noindent See \Cref{sec:proof-minimax-lower-bound} for the proof.
\medskip

Taken together, \Cref{prop:upper-rate} and
\Cref{thm:new-lower-bound} identify the same benchmark-dependent local
modulus as governing both the upper and lower bounds. The two results differ
only in the noise scale: the upper bound uses
$
\tau_{n,\delta}\asymp
\frac{\VC(\cS)\log n+\log(1/\delta)}{n},
$
whereas the lower bound for the half-interval model uses the parametric scale
\(1/n\). More explicitly, using convexity of \(\fnull\), one obtains
\[
\Delta_{\fnull}(\alpha; \tau_{n,\delta})
\;\leqslant\;
4\bigl[\VC(\cS)\log(2n)+\log(32/\delta)\bigr]\,
\Delta_{\fnull}(\alpha;1/n).
\]
Thus the proposed procedure is rate-sharp
up to logarithmic factors in the half-interval/MLR model, uniformly over
benchmark trade-off functions \(\fnull\).

\subsection{Consequences for common benchmark curves}

We now illustrate how the same modulus produces different local rates depending
on the geometry of the benchmark curve. To simplify the discussion, we suppress
logarithmic factors and write \(\tau_{n,\delta}\asymp \VC(\mathcal S)/n\).

\begin{example}[Identity testing]
\label{ex:gap-linear}
Testing $P=Q$ is equivalent to setting  $\fnull(\alpha)=1-\alpha$. In this case, we have  
\begin{equation}
\Delta_{\fnull}(\alpha;\tauND)
=
\bigl[\tauND+\sqrt{\tauND V(\alpha)}\bigr]\wedge (1-\alpha)
+
\Bigl(\sqrt{\tauND V(1-\alpha)}\Bigr)\wedge (1-\alpha).
\label{eq:Delta-linear}
\end{equation}
This yields a locally adaptive behavior:
\begin{itemize}
\item if \(\alpha\) stays in the interior of \([0,1]\), then
\(V(\alpha)\asymp 1\) and \(V(1-\alpha)\asymp 1\), so
\(\Delta_{\fnull}(\alpha;\tauND)\asymp \sqrt{\tauND}\asymp \sqrt{\frac{\VC(\cS)}{n}}\) ;
\item if \(\alpha\) is close to \(0\), then
\[
\Delta_{\fnull}(\alpha;\tauND)
\asymp
\tauND+\sqrt{\tauND\alpha}.
\]
In particular, at the boundary \(\alpha=0\), and more generally for
\(\alpha\lesssim \tauND\), the sufficient gap improves to order
\(\tauND \asymp \frac{\VC(\cS)}{n}\).
\end{itemize}
Thus, for the linear benchmark, the sufficient gap is of order
\(\widetilde O(n^{-1/2})\) at interior points, while near the left
boundary it can improve to a smaller order \(\widetilde O(n^{-1})\).
\end{example}

\begin{example}[Tolerant total-variation benchmark]
\label{ex:gap-tv} 
Recall that the total-variation benchmark
\(\dTV(P,Q)\leqslant \varepsilon\) corresponds to
\(\fnull(\alpha)=g_\varepsilon(\alpha)=(1-\varepsilon-\alpha)_+\). 
In this case, we have
\begin{equation}
\Delta_{\fnull}(\alpha;\tauND)
=
\bigl[\tauND+\sqrt{\tauND V(\alpha)}\bigr]\wedge (1-\varepsilon-\alpha)_+
+
\Bigl(\sqrt{\tauND V((1-\varepsilon-\alpha)_+)}\Bigr)\wedge (1-\varepsilon-\alpha)_+.
\label{eq:Delta-tv}
\end{equation}
This benchmark behaves differently from the identity benchmark near
\(\alpha=0\). At the boundary,
\[
\Delta_{\fnull}(0;\tauND)
=
\tauND\wedge(1-\varepsilon)
+
\Bigl(\sqrt{\tauND\,V(1-\varepsilon)}\Bigr)\wedge(1-\varepsilon).
\]
Thus, for fixed \(\varepsilon\in(0,1)\) bounded away from \(0\) and \(1\), the
dominant term is of order \(\sqrt{\tauND}\). More generally, the boundary
behavior depends on \(V(1-\varepsilon)=\min\{\varepsilon,1-\varepsilon\}\).
In the usual tolerant regime with fixed nondegenerate \(\varepsilon\), there is
therefore no \(n^{-1}\)-type improvement at \(\alpha=0\), in stark contrast to
the identity benchmark.
\end{example}

\begin{example}[Curved benchmarks with slow boundary rates]
\label{ex:gap-sqrt}
The boundary rate near \(\alpha=0\) can be arbitrarily slow. Let \(\rho\geqslant 1\)
and consider
\[
\fnull(\alpha)=(1-\alpha^{1/\rho})^\rho,\qquad \alpha\in[0,1].
\]
This is a valid trade-off function.

With some calculations, we can show that at interior points, for instance \(\alpha\in[c,1-c]\) with constant
\(c>0\), both the input-side and output-side variance factors are of constant
order, and \(\fnull\) has derivative bounded away from zero and infinity.
Therefore
\[
\Delta_{\fnull}(\alpha;\tauND)\asymp \sqrt{\tauND}.
\]
However, near $0$, the behavior is different. At \(\alpha=0\),
\[
\Delta_{\fnull}(0;\tauND)
=
1-(1-\tauND^{1/\rho})^\rho
\asymp \tauND^{1/\rho}.
\]
The same order holds more generally for \(0\leqslant \alpha\lesssim \tauND\).
Thus, since \(\tauND=\widetilde O(n^{-1})\), the sufficient gap near
\(\alpha=0\) is of order \(\widetilde O(n^{-1/\rho})\), which can be made
arbitrarily slow by taking \(\rho\) large. This illustrates that even for smooth
curved benchmarks, the detection boundary is governed by the local geometry of
\(\fnull\), rather than by a single global smoothness parameter.
\end{example}

\subsection{Application to tolerant total-variation testing}

The tolerant-TV benchmark \(g_\varepsilon(\alpha)=(1-\varepsilon-\alpha)_+\)
also gives a direct comparison with the classical discrete tolerant-testing
literature. We record the implication of \Cref{prop:upper-rate} for
an alphabet of size \(k\).

\begin{corollary}[Tolerant TV testing on a finite alphabet]
\label{cor:tvtolerant-vc}
Consider the tolerant testing problem~\eqref{eq:tolerant-testing} over an
alphabet of size \(k\), and take \(\cS=2^{[k]}\), so that
\(\VC(\cS)=k\). Define
\[
\tau_{n,\delta}
=
\frac{4\{k\log(2n)+\log(32/\delta)\}}{n}.
\]
There exists a universal constant \(C>0\) such that if
\[
\varepsilon_2-\varepsilon_1
\;\geqslant\;
C\sqrt{\tau_{n,\delta}},
\]
then the test \(\psi_{n,\delta}(\cS)\) from \Cref{thm:general}
satisfies
\[
\errone(\psi_{n,\delta})\leqslant \delta,
\qquad
\errtwo(\cM_{\cS};\psi_{n,\delta})\leqslant \delta.
\]
\end{corollary}

To compare \Cref{cor:tvtolerant-vc} with the discrete
tolerant-testing literature, it is helpful to translate known
sample-complexity bounds into detectable-gap scalings. The main result of
\cite{price2021tolerant} implies, up to logarithmic factors, that on a
discrete domain of size \(k\), the noiseless regime \(\varepsilon_1=0\)
admits detection at separation \(\varepsilon_2 \asymp k^{1/4}n^{-1/2}\),
whereas in the genuinely tolerant regime
\(\varepsilon_1/\varepsilon_2=\Theta(1)\), the detectable gap satisfies
$
\Delta:=\varepsilon_2-\varepsilon_1
\asymp
\widetilde\Theta\!\left(\sqrt{k/n}\right).
$
By \Cref{cor:tvtolerant-vc}, our general framework yields the
sufficient condition
$
\Delta \gtrsim \widetilde O\!(\sqrt{k/n}).
$
Thus we recover the same worst-case gap scaling. On the other hand, our
result does not capture the sharper \(k^{1/4}n^{-1/2}\) scaling available in
the special noiseless discrete setting. This is expected: our theorem is a
general sufficient condition for trade-off testing over a VC witness class,
rather than a sharp minimax analysis tailored to discrete total-variation
testing.

\section{Beyond exact attainability}
\label{sec:misspecification}

The exact attainability framework in \Cref{sec:general} gives a clean
possibility theory: if the Neyman--Pearson witness sets lie in a finite-VC
class \(\mathcal S\), then the trade-off curve can be tested from finite
samples, and if \(\VC(\mathcal S)=\infty\), exact uniform testing is impossible.
In applications, however, exact attainability may be too rigid. The true
likelihood-ratio level sets may not belong to any tractable VC class, even
though they can be well approximated by one.

This section extends the theory to that setting. The resulting picture is a
bias--variance trade-off. Enlarging \(\mathcal S\) improves approximation of the
true Neyman--Pearson witnesses, but increases the estimation error through
\(\VC(\mathcal S)\). Approximate attainability makes this tradeoff explicit
while preserving the assumption-free type~I validity of the test.

\subsection{Approximate attainability}

We begin by formalizing the approximation model.

\begin{definition}[$(\mathcal S, \eta_1, \eta_2)$-attainability]\label{def:approx-attain}
The class of $(\mathcal S, \eta_1, \eta_2)$-attainable pairs is defined as
\[
\cM_{\mathcal S, \eta_1, \eta_2}
=
\big\{(P,Q) : \exists\, (P', Q') \in \cM_{\mathcal S}
\text{ with }
\dTV(P, P') \leqslant \eta_1,\;
\dTV(Q, Q') \leqslant \eta_2\big\}.
\]
\end{definition}
Thus \((P,Q)\) need not have likelihood-ratio level sets in \(\mathcal S\)
itself; it only needs to be close, in total variation, to a pair that does.
The parameters \(\eta_1\) and \(\eta_2\) measure the approximation error in the
two coordinates of the trade-off curve.
Total variation is the natural metric here because it directly
controls the trade-off function. Indeed, whenever
\(\dTV(P,P')\leqslant \eta_1\) and \(\dTV(Q,Q')\leqslant \eta_2\), one has
\[
\Toff{P}{Q}(\alpha)\geqslant \Toff{P'}{Q'}(\alpha+\eta_1)-\eta_2.
\]
Thus, closeness in total variation to an exactly attainable pair translates
immediately into pointwise control of the trade-off function. The exact model
is recovered by taking \(\eta_1=\eta_2=0\).

The next theorem shows that the test from \Cref{thm:general} degrades
gracefully under this approximation. Type~I error remains assumption-free,
because the rejection rule is unchanged. Power is affected only through
additional shifts in the separation condition, corresponding to the
approximation errors in the \(P\)- and \(Q\)-coordinates.
\begin{theorem}\label{thm:misspecified}
Fix $\delta \in (0,1)$ and a set collection $\cS$ with
$\VC(\cS) < +\infty$. The test $\psi_{n,\delta}$ defined in
\eqref{eq:general-test} satisfies the following.

\begin{enumerate}[label=(\roman*)]
\item \textbf{Type~I error control (assumption-free).}
$\errone(\psi_{n,\delta}) \;\leqslant\; \delta.
$

\item \textbf{Type~II error control (under approximate attainability).}
If
\begin{equation}\label{eq:gap-misspecified}
\forall\, \alpha \in [0,1]:\qquad
\falt(\alpha)
\;\leqslant\;
\max\big\{\hinvinv\circ\fnull\circ \hupup(\alpha + 2\eta_1)-2\eta_2,\;0\big\},
\end{equation}
then 
$\errtwo(\cM_{\cS,\eta_1,\eta_2};\, \psi_{n,\delta}) \;\leqslant\; \delta.
$
\end{enumerate}
\end{theorem}

\noindent See \Cref{sec:proof-misspecified} for the proof.
\medskip

Compared with \Cref{thm:general}, approximate attainability only
inflates the required separation by the approximation errors \(\eta_1\) and
\(\eta_2\). Thus the total detection threshold has two components:
\[
\text{estimation error from }\VC(\mathcal S)
\qquad+\qquad
\text{approximation error from }(\eta_1,\eta_2).
\]
This
decomposition parallels the usual bias--variance trade-off in nonparametric
estimation: the statistician chooses the model class \(\cS\) to balance
approximation quality against the complexity cost of uniform estimation.

\subsection{Application to log-concave distributions}

We illustrate approximate attainability on univariate log-concave
distributions. This class is broad enough that exact finite-VC attainability is
not available uniformly: likelihood-ratio level sets of log-concave pairs can
oscillate arbitrarily often. Nevertheless, log-concavity imposes enough shape
structure that these level sets can be approximated by unions of intervals.
The approximate-attainability theorem then yields finite-sample testing
guarantees by balancing the number of intervals against the approximation
error.

\begin{definition}
A distribution on \(\bbR\) is \emph{log-concave} if it admits a density \(p\)
such that \(\log p\) is concave on its support.
\end{definition}

Exact finite-VC attainability cannot hold uniformly over the full log-concave
class. Indeed, one can construct log-concave densities
\begin{align*}
p_\pm(x;\theta,s)
=
\frac{1}{Z_\pm\cdot s}
\exp\left(-\frac{(x-\theta)^2}{2s^2}\pm a\cos\left(\frac{x-\theta}{s}\right)\right),
\qquad
\theta\in\bbR,\ s>0,
\end{align*}
for some \(a\in(0,1)\), where \(Z_\pm\) are normalizing constants depending
only on \(a\). Note that
\begin{align*}
\frac{\partial^2\log p_\pm}{\partial x^2}(x;\theta,s)
=
\frac{-1\mp a\cos((x-\theta)/s)}{s^2}
<0.
\end{align*}
Therefore, both \(p_\pm(\cdot;\theta,s)\) are log-concave for all
\((\theta,s)\). However, the likelihood ratio
\begin{align*}
\frac{ p_+}{p_-}(x;\theta,s)
=
\frac{Z_-}{Z_+}\exp\left(2a\cos\left(\frac{x-\theta}{s}\right)\right)
\end{align*}
is periodic and fluctuates between
\((Z_-/Z_+)\cdot [e^{-2a},e^{2a}]\). Thus, for any threshold
\(t\in(Z_-/Z_+)\cdot (e^{-2a},e^{2a})\), the collection of level sets
\[
\{x:\, p_+/ p_-(x;\theta,s)>t\}
\]
can shatter arbitrarily many finite points on \(\bbR\) by suitably choosing
the shift \(\theta\) and the scale \(s\).

Such complicated level-set structures can nevertheless be handled from an
approximation-theoretic viewpoint. Our approximating classes are the
interval-union classes
\begin{align}\label{eq:union-interval}
   \cI_K
:=
\left\{
\bigcup_{j=1}^K (a_j,b_j)
:\;
-\infty \leqslant a_1 \leqslant b_1 \leqslant \cdots \leqslant a_K \leqslant b_K \leqslant \infty
\right\}. 
\end{align}
We allow degenerate intervals $a_j=b_j$, so $\cI_K$ is exactly the class of
unions of at most $K$ intervals. In particular,
\[
\cI_1 \subseteq \cI_2 \subseteq \cdots,
\]
and $\cI_1$ contains half-lines. 
This class has VC dimension $\mathrm{VC}(\cI_K)=2K$.
The following approximation result shows that log-concave pairs are
approximately attainable by interval-union classes.

\begin{lemma}\label{lem:log-concave-approx-attain}
For any \(\eta>0\), every pair of log-concave distributions on \(\bbR\) is
\((\cI_{\lceil \frac{1}{\sqrt{\eta}}\rceil}, \widetilde{O}(\eta), \widetilde{O}(\eta))\)-attainable. That is,
\[
\cM_{\mathrm{LC}}
\define
\big\{(P,Q): P,Q \text{ are log-concave on } \bbR\big\}
\subseteq
\cM_{\cI_{\lceil \frac{1}{\sqrt{\eta}}\rceil}, \widetilde{O}(\eta), \widetilde{O}(\eta)}.
\]
\end{lemma}

\noindent The underlying
mechanism is geometric: log-concavity imposes strong shape constraints on the
densities, which in turn limit the complexity of their likelihood-ratio level
sets. See \Cref{sec:log-concave} for the proof.
\medskip

We now specialize the general theory to tolerant total-variation testing.

\begin{corollary}\label{cor:log-concave-tolerant}
Consider the tolerant testing problem~\eqref{eq:tolerant-testing} based on
\(n\) samples from each distribution. Fix a target level \(\delta \in (0,1)\).
The test \(\psi_{n,\delta}\) defined
in \eqref{eq:general-test} based on the class \(\cI_{ \lceil n^{1/5}\rceil}\) achieves \(\errone(\psi_{n,\delta}) \leqslant \delta\) and \(\errtwo(\cM_{\mathrm{LC}};\, \psi_{n,\delta}) \leqslant \delta\), 
provided that the gap satisfies
\begin{align*}
\varepsilon_2-\varepsilon_1 \geqslant \widetilde{O}(n^{-2/5}).
\end{align*}
\end{corollary}

\noindent The rate follows from the bias--variance balance described above. With \(K\)
intervals, the VC penalty contributes an estimation term of order
\(\sqrt{K/n}\) for the tolerant-TV benchmark, while
\Cref{lem:log-concave-approx-attain} contributes approximation error
\(\widetilde O(K^{-2})\). Choosing \(K\asymp n^{1/5}\) balances these terms and
yields the rate \(\widetilde O(n^{-2/5})\). See~\Cref{sec:proof-cor-log-concave} for the proof. 
\medskip

For the special non-tolerant closeness problem \(\varepsilon_1=0\), sharper
problem-specific methods achieve the minimax rate
\(\widetilde O(n^{-4/9})\) for univariate log-concave distributions
\cite{TV_log_concave}. Our \(\widetilde O(n^{-2/5})\) rate is slower, as
expected, because it comes from a generic trade-off testing procedure and
applies to arbitrary benchmark curves, not only to total-variation closeness.

\section{Efficient computation}\label{sec:computation}

The statistical results above define the test as a scan over a class of witness
sets \(\cS\):
\begin{equation}
\label{eq:test-comp-start}
\psi_{n,\delta}(\cS)
=
\mathbf 1\Bigl\{
\exists S\in\cS:
h_+\bigl(Q_n(S^\complement)\bigr)
<
\fnull\bigl(h_+(P_n(S))\bigr)
\Bigr\}.
\end{equation}
At first sight, this may appear computationally difficult, because \(\cS\) may
be infinite and the rejection condition is nonlinear in the empirical
probabilities. In this section, we show that the scan has a simple algorithmic
reduction: it can be implemented using a finite number of cost-sensitive
empirical risk minimization problems over \(\cS\).

This reduction separates the statistical and computational aspects of the
method. Finite VC dimension controls the statistical error, but it does not by
itself imply computational tractability. Computation is efficient whenever the
corresponding cost-sensitive ERM problem over \(\cS\) is efficient.

Recall that \(\hinv\) is the left-continuous inverse of \(h_+\). Define the
surrogate benchmark
\[
\widetilde \fnull
=
\hinv\circ \fnull\circ h_+.
\]
By \Cref{lem:hinv-properties}(iii), the rejection rule in
\eqref{eq:test-comp-start} is equivalent to
\[
\exists\,S\in\cS:
\qquad
Q_n(S^\complement)<\widetilde\fnull(P_n(S)).
\]
Moreover, by \Cref{lem:tilde-f}, \(\widetilde\fnull\) is convex and
non-increasing.

Since \(P_n(S)\in\{0,1/n,\ldots,1\}\), only the values of
\(\widetilde\fnull\) on this grid matter. For \(k=1,\ldots,n\), define
\begin{align}
\label{eq:intercept_and_slope}
\lambda_k
&=
n\left[
\widetilde{\fnull}\left(\frac{k}{n}\right)
-
\widetilde{\fnull}\left(\frac{k-1}{n}\right)
\right],
\\
c_k
&=
\widetilde{\fnull}\left(\frac{k-1}{n}\right)
-
(k-1)
\left[
\widetilde{\fnull}\left(\frac{k}{n}\right)
-
\widetilde{\fnull}\left(\frac{k-1}{n}\right)
\right].
\nonumber
\end{align}
Equivalently, \(c_k+\lambda_k t\) is the affine extension of the linear segment
joining the two grid values of \(\widetilde\fnull\) at \((k-1)/n\) and \(k/n\).
Convexity of \(\widetilde\fnull\) implies that its piecewise-linear
interpolation on the grid is the maximum of these affine pieces.

\begin{proposition}[Reduction to weighted ERM]\label{prop:computation-main}
The test \(\psi_{n,\delta}(\cS)\) rejects if and only if there exists
\(k\in\{1,\ldots,n\}\) such that
\begin{equation}
\label{eq:test-cover}
\inf_{S\in\cS}
\left\{
Q_n(S^\complement)-\lambda_k P_n(S)
\right\}
<
c_k.
\end{equation}
\end{proposition}

Thus the nonlinear scan in \eqref{eq:test-comp-start} reduces to at most \(n\)
linear optimization problems over the same witness class \(\cS\). For fixed
\(k\), since \(\widetilde\fnull\) is non-increasing, we have
\(\lambda_k\leqslant 0\). Hence the objective in~\eqref{eq:test-cover} can be
written as
\[
Q_n(S^\complement)-\lambda_k P_n(S)
=
Q_n(S^\complement)+(-\lambda_k)P_n(S).
\]
If \(S\) is interpreted as the region classified as coming from \(Q\), then
\(Q_n(S^\complement)\) is the empirical false-negative rate and \(P_n(S)\) is
the empirical false-positive rate. Thus each subproblem is a cost-sensitive
classification problem over \(\cS\), with false-negative cost \(1\) and
false-positive cost \(-\lambda_k\). The different values of \(k\) correspond to
the different slopes of the corrected benchmark, and hence to different
relative costs of false positives and false negatives.

Consequently, if each cost-sensitive ERM problem over \(\cS\) can be solved in
time \(\mathcal T(n)\), then the full test can be implemented in
\(O(n\mathcal T(n))\) time, up to preprocessing costs.

For specific set classes, the weighted ERM problem admits efficient algorithms.
For unions of at most \(K\) intervals, \(\cI_K\), each subproblem reduces to
finding a union of \(K\) intervals on the sorted pooled sample that minimizes
the corresponding cost-sensitive empirical risk, or equivalently maximizes the
associated weighted discrepancy. This can be solved in \(O(nK)\) time by
dynamic programming, with faster specialized algorithms for maximum-sum
interval problems available in some settings; see
\cite{csuros2004maximum,bengtsson2006computing}. For halfspaces in
\(\bbR^d\), the subproblem is weighted linear classification over halfspaces.
In fixed dimension, one can enumerate the combinatorially distinct halfspaces
induced by the pooled sample, yielding a polynomial-time algorithm in \(n\) for
fixed \(d\); the exponent depends on \(d\).

\begin{remark}[The unnormalized margin as a special case]\label{rem:unnorm-special-case}
Replacing the variance-adaptive margin \(\hup\) by the uniform DKW-type shift
\(t \mapsto t + \etaND\) amounts to setting
\[
\widetilde{f}_0(t) = \fnull(t + \etaND) - \etaND,
\]
which is already piecewise-linear whenever \(\fnull\) is. The entire argument
above then applies verbatim. In particular, the variance-adaptive margin
incurs no additional computational cost relative to the simpler uniform
margin.
\end{remark}

\section{Numerical experiments}\label{sec:experiments}

In this section, we perform experiments to illustrate the finite-sample behavior of the proposed procedures in three settings. We first study type~I error and power in the monotone likelihood ratio (MLR) model, where the theory is most transparent. We then examine the simultaneous confidence bands in the same setting. Finally, we study testing over the union-of-intervals classes $\mathcal I_K$, highlighting the power of the adaptive procedure.

\subsection{Type~I error and power in the MLR model}\label{sec:exp-mlr}

\paragraph{Data-generating mechanisms.}
  We consider two location families satisfying the MLR condition: 
  (a) Gaussian shift, $P=N(0,1)$ and $Q=N(\mu',1)$;
  and (b) Laplace shift, $P=\mathrm{Lap}(0,1)$ and $Q=\mathrm{Lap}(\mu',1)$.
  Let $\mu_0 = 1$. We consider a panel-specific benchmark $f_{0}$
  matched to the data-generating family: $f_{0}=G_{\mu_{0}}$ in the Gaussian
  panel and $f_{0}=L_{\mu_{0}}$ in the Laplace panel. Here
  $G_{\mu}(\alpha)=\Phi\!\bigl(\Phi^{-1}(1-\alpha)-\mu\bigr)$ is the
  $\mu$-Gaussian differential-privacy trade-off function of \cite{GDP}, and
  $L_{\mu}$ is the trade-off function of the Laplace shift
  $(\mathrm{Lap}(0,1),\mathrm{Lap}(\mu,1))$.

  The shift parameter in the data-generating process ranges over
  $\mu'\in\{0.0,0.5,0.8,1.0\}$ under the null and
  $\mu'\in\{1.2,1.5,2.0\}$ under the alternative. We vary the sample size $n$ from $100$ to $5000$.

  \begin{figure}[t]
    \centering
    \subfloat[$P=N(0,1)$, $Q=N(\mu',1)$]{%
      \includegraphics[width=0.48\textwidth]{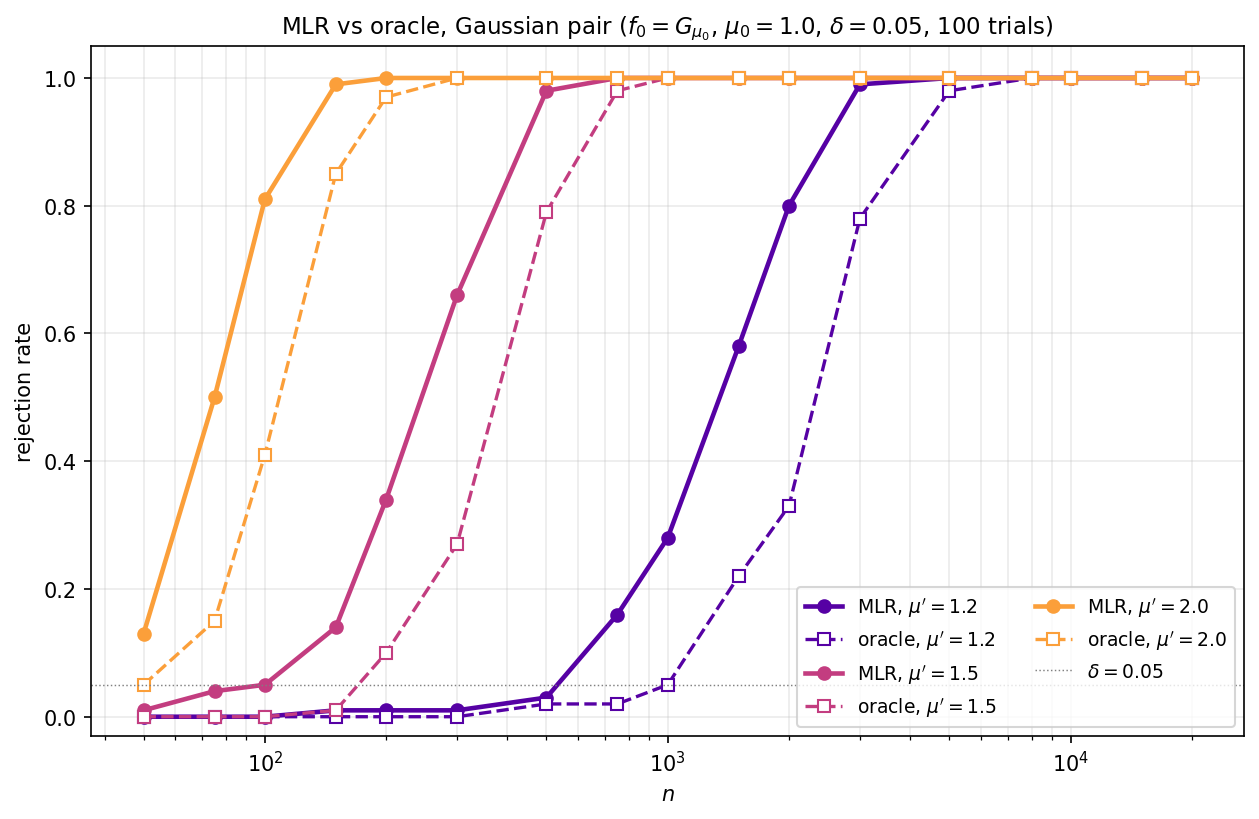}%
      \label{fig:exp-mlr-gauss}}
    \hfill
    \subfloat[$P=\mathrm{Lap}(0,1)$, $Q=\mathrm{Lap}(\mu',1)$]{%
      \includegraphics[width=0.48\textwidth]{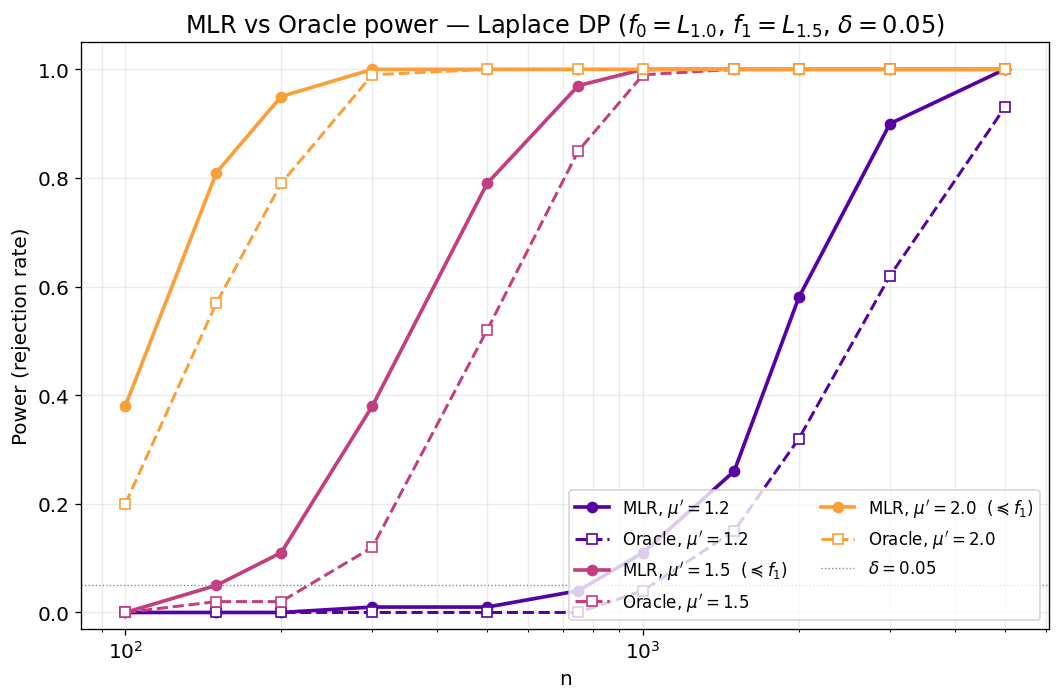}%
      \label{fig:exp-mlr-lap}}
    \caption{Empirical power of the MLR test (solid, filled circles) and the
      oracle test (dashed, open squares) against shift alternatives in the MLR
      model. The benchmark is $f_{0}=G_{\mu_{0}}$ in panel~(a) and $L_{\mu_{0}}$
      in panel~(b), with $\mu_{0}=1$ and $\delta=0.05$; each marker averages
      $100$ replications. Empirical type~I error is zero in every null cell and
      is omitted.}
    \label{fig:exp-mlr}
  \end{figure}

\paragraph{Procedures.}
  We compare two level-$\delta=0.05$ procedures: the \emph{MLR test} from
  \Cref{thm:mlr}, which scans the empirical Neyman--Pearson curve over
  $\mathcal{S}_{\mathrm{MLR}}$ with DKW margin
  $\eta_{n,\delta}=\sqrt{\log(4/\delta)/(2n)}$, and an \emph{oracle-aided test} that
  is given the most-violating witness\footnote{This oracle-aided test can be viewed as the idealized version of the test proposed in~\cite{DPAuditing}.}
  \[
  S^{\star}\in\arg\max_{S\in\mathcal{S}_{\mathrm{MLR}}}
  \bigl\{f_{0}\!\bigl(P(S)\bigr)-Q(S^{\complement})\bigr\}
  \]
  (equivalently, the most-violating likelihood-ratio threshold $t^{\star}$) and
  applies the one-witness corrected test
  \[
  Q_{n}\!\bigl((S^{\star})^{\complement}\bigr)+\eta_{1}
  \;<\;f_{0}\!\bigl(P_{n}(S^{\star})+\eta_{1}\bigr),
  \qquad
  \eta_{1}=\sqrt{\log(2/\delta)/(2n)}.
  \]
  The oracle is hypothetical and serves only as a benchmark. Each cell is
  replicated $100$ times. 

  \paragraph{Findings.}
  Across all null cells, both procedures are conservative: the empirical type~I
  error is zero throughout the simulation grid. Under the alternative, the MLR
  test shows a clear transition from low to high power as either $n$ or $\mu'$
  increases, and its performance is close to that of the oracle benchmark; see
  \Cref{fig:exp-mlr}. In some intermediate regimes, it is even slightly
  more powerful, which is consistent with the fact that it searches over the
  whole empirical Neyman--Pearson curve rather than comparing at a single level.
  The Laplace experiment shows the same qualitative behavior.

\subsection{Confidence bands in the MLR model}\label{sec:exp-mlr-band}

We next turn from testing to estimation and illustrate the simultaneous
confidence band of \Cref{cor:confidence-bands} specialized to
$\mathcal{S}_{\mathrm{MLR}}$, for which
$\mathrm{VC}(\mathcal{S}_{\mathrm{MLR}})=2$. The band is computed from the
witness pairs $\bigl(P_n(S),Q_n(S^\complement)\bigr)$ obtained by sweeping the
pooled order statistics, with corrections $h_{\pm}$ at level $\delta=0.05$.

\paragraph{Design.}
We reuse the Gaussian and Laplace shift families from
\Cref{sec:exp-mlr}, both with $\mu=1$. Sample sizes are
$n\in\{5\,000,\,10\,000,\,20\,000\}$, chosen large enough that the band is
visibly informative on the scale of the curve, and coverage is estimated
from $200$ Monte Carlo replications per cell on the grid
$\alpha\in\{0,0.005,\ldots,1\}$.

\begin{figure}[t]
  \centering
  \includegraphics[width=0.97\textwidth]{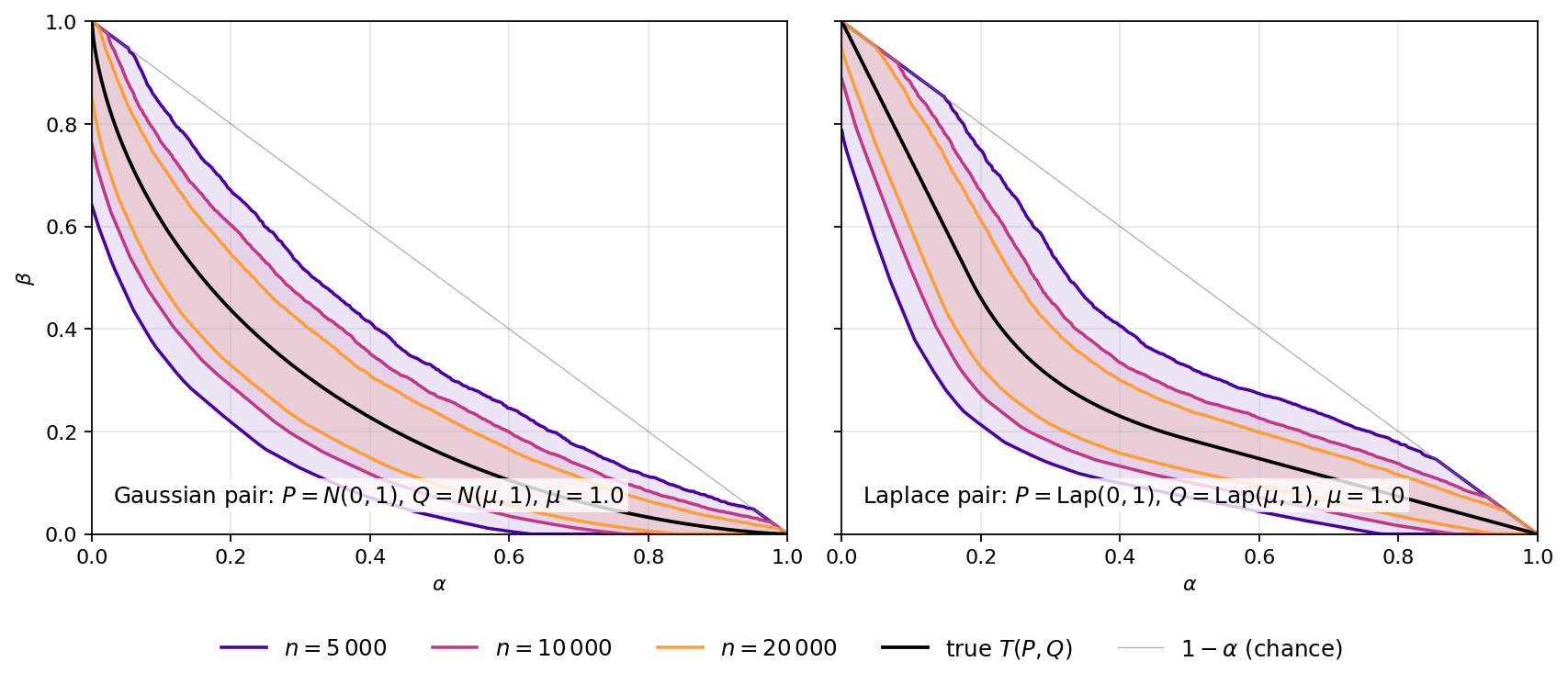}
  \caption{Simultaneous $(1-\delta)$-confidence bands for $\Toff{P}{Q}$ in the
    MLR setting, computed on a single sample from each pair. The shaded
    region is $[\hat L_{\mathrm{GCM}}(\alpha),\,\hat U(\alpha)]$; the
    dashed/dotted lines are the upper and lower envelopes; black is the
    analytical $T(P,Q)$. As~$n$ grows, the correction shrinks and the band
    concentrates around the true curve.}
  \label{fig:exp-mlr-band-demo}
\end{figure}

\paragraph{Findings.}
\Cref{fig:exp-mlr-band-demo} shows that the bands narrow with
increasing $n$ and are widest in the interior of the curve, as predicted by
the variance-adaptive correction. Across all $1{,}200$ replications the empirical
simultaneous coverage of
$\hat L_{\mathrm{GCM}}\leqslant T(P,Q)\leqslant\hat U$ equals $1.000$ in
every cell, above the nominal $1-\delta=0.95$. The conservatism is
inherited from the union bound underlying $\tau_{n,\delta}$ in
\Cref{lem:nuc}; sharpening the constants is left to future work.

\subsection{Testing under unknown union-of-intervals complexity}\label{sec:exp-iu}

We next study the union-of-intervals classes $\mathcal I_K$, for which
$\mathrm{VC}(\mathcal I_K)=2K$. If $(P,Q)$ is
$\mathcal I_{K^\star}$-attainable, then using $K<K^\star$ introduces bias
(no $S\in\mathcal I_K$ can witness the violation), while using $K>K^\star$
enlarges the margin $\tau_{n,\delta}$ and reduces power without any
expressive gain. The experiment below illustrates this bias--variance
trade-off and evaluates the adaptive procedure of
\Cref{cor:adaptive}.

\paragraph{Data-generating mechanism.}
We take $P=\mathrm{Unif}[0,1]$ and let $Q$ have piecewise-constant density
\[
q(x)=
\begin{cases}
1+\delta_q,& x\in\bigl[\tfrac{2j}{2K^\star},\,\tfrac{2j+1}{2K^\star}\bigr],\\
1-\delta_q,& x\in\bigl[\tfrac{2j+1}{2K^\star},\,\tfrac{2j+2}{2K^\star}\bigr],
\end{cases}
\qquad j=0,\ldots,K^\star-1,
\]
with $K^\star=2$, so the optimal witness is a union of two intervals. The
benchmark is $\fnull=g_\varepsilon$ with $\varepsilon=0.05$, and we sweep
$\delta_q\in\{0.85,0.95\}$ over
$n\in\{1000,1500,\ldots,5000,6000,7000,8000\}$ and $K\in\{1,2,3,4\}$. Each
cell is replicated $50$ times in the fixed-$K$ sweep and $100$ times in the
adaptive comparison.

\begin{figure}[t]
  \centering
  \includegraphics[width=0.85\textwidth]{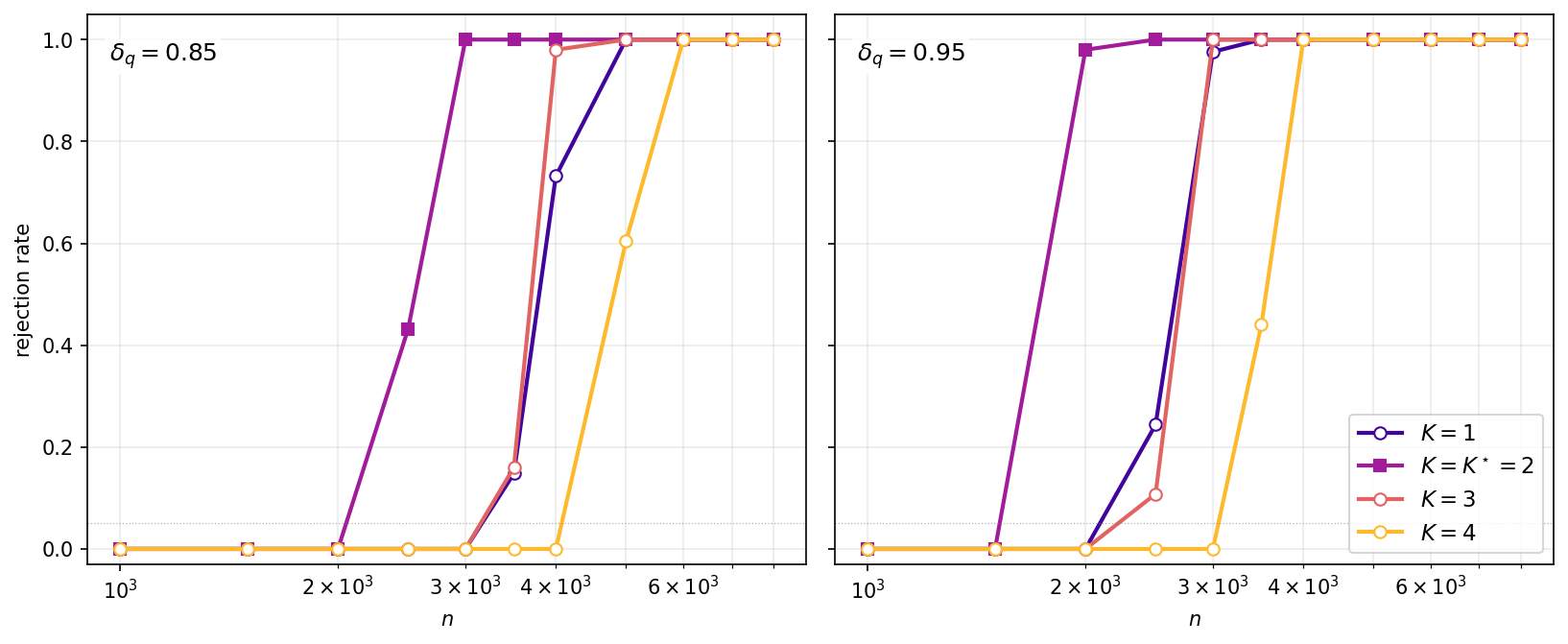}
  \caption{Effect of $K$ in the dual interval-union test for the
    piecewise-constant uniform pair, $K^{\star}=2$. $P=\mathrm{Unif}[0,1]$,
    $Q$ has two bumps of amplitude $\delta_{q}$, and the benchmark is
    $f_{0}=g_{\varepsilon}$ with $\varepsilon=0.05$. The matched choice
    $K=2$ saturates earliest; $K=1$ pays a bias cost, while $K\in\{3,4\}$
    pay a VC penalty.}
  \label{fig:exp-iu-pc2}
\end{figure}

\begin{figure}[htbp]
  \centering
  \includegraphics[width=0.97\textwidth]{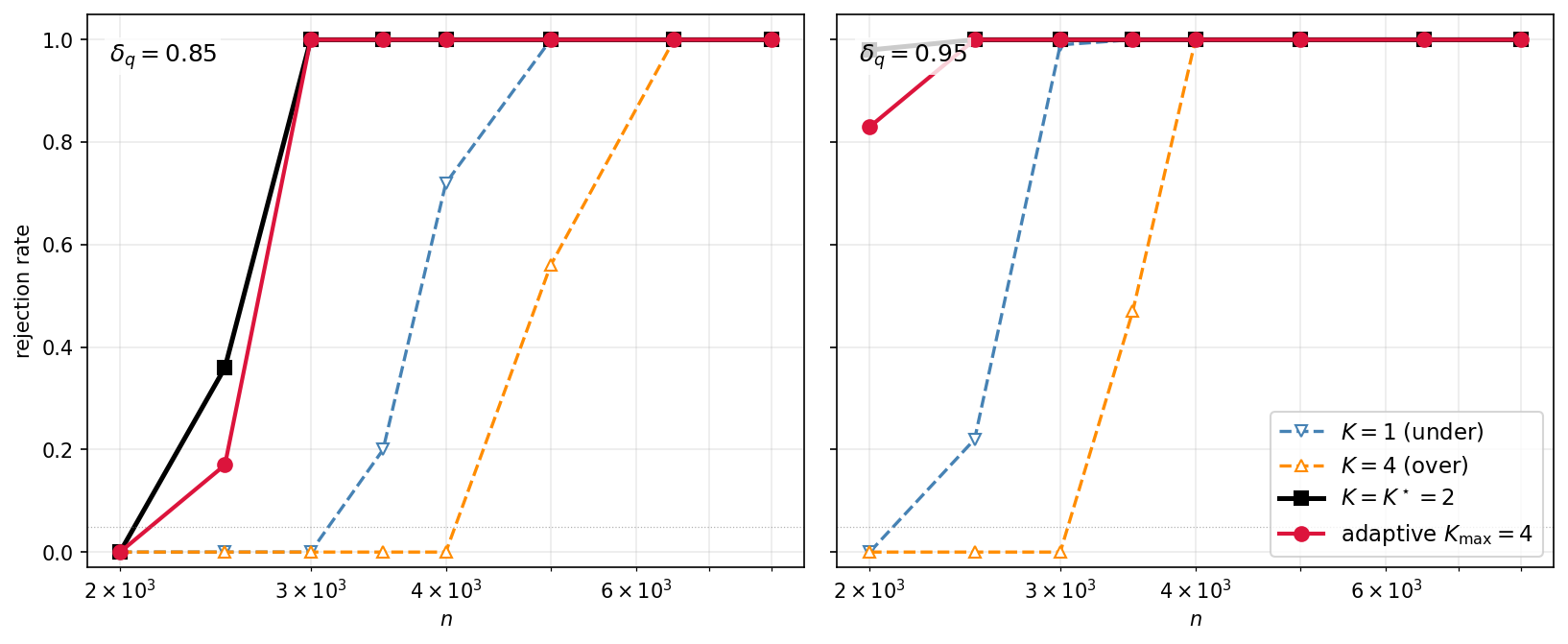}
  \caption{Adaptation under unknown union-of-intervals complexity
    ($K^{\star}=2$). Rejection rate against~$n$ for the matched test
    $K=2$, the under-specified test $K=1$, the over-specified test $K=4$,
    and the adaptive Bonferroni procedure of \Cref{cor:adaptive}
    with $K_{\max}=4$. The adaptive curve closely matches the oracle
    $K^{\star}$ curve in both panels.}
  \label{fig:exp-iu-adaptive}
\end{figure}

\paragraph{Findings for fixed $K$.}
Both bias and variance are visible in \Cref{fig:exp-iu-pc2}: the
matched choice $K=K^\star=2$ saturates first; $K=1$ pays a bias cost as it
cannot witness the union of bumps; $K\in\{3,4\}$ pay a VC penalty without
expressive gain.

\paragraph{Adaptation when $K^\star$ is unknown.}
We instantiate \Cref{cor:adaptive} with $K_{\max}=4$ on the
$K^\star=2$ piecewise-constant pair, and compare it with the matched test
$K=2$, the under-specified test $K=1$, and the over-specified test $K=4$ on
the same Monte Carlo draws. \Cref{fig:exp-iu-adaptive} shows that the
adaptive curve closely tracks the matched curve and clearly improves over the
mismatched fixed-$K$ baselines. The finite-sample adaptation cost is therefore
small in this regime.

\section{Related work}\label{sec:related-work}
Our work is most closely related to two lines of research: privacy auditing and tolerant distribution testing.

\paragraph{Privacy auditing.}
Finite-sample methods for auditing $(\varepsilon,\delta)$-DP typically
construct or select critical events in the privacy
definition~\cite{Auditing1,Auditing2,Auditing3,DP-Definition}, but the
$(\varepsilon,\delta)$ formulation is discontinuous: statistically
indistinguishable output distributions can yield opposite privacy
conclusions~\cite{Auditing-Impossibility}.  Recent work proposes auditing under
$f$-DP via density or likelihood-ratio
estimation~\cite{empirical-auditing-f-dp-1, empirical-auditing-f-dp-2,
DPAuditing}, avoiding this discontinuity but primarily providing asymptotic
guarantees for fixed distribution pairs. While these approaches establish
consistency of the estimated trade-off curve and, in some cases, pointwise
convergence rates of order $\widetilde{O}(n^{-1/2})$ for individual trade-off points~\cite{DPAuditing}, this rate
corresponds to the parametric regime in our setting when $\VC(\mathcal S)=O(1)$
and thus matches our bound. However, these results remain pointwise and do not
yield uniform non-asymptotic guarantees over the entire curve or a minimax
characterization. In contrast, our test targets the trade-off function directly
and provides non-asymptotic error control uniform over a nonparametric class,
with assumption-free type~I control that prevents false alarms even under model
misspecification.

\paragraph{TV testing.}
For discrete distributions on $[d]$, sharp minimax rates are known for identity and closeness testing under total variation~\cite{TV2008,TV2000,chan2014optimal,diakonikolas2019collision}. The tolerant testing framework~\cite{tolerant2001,tolerant2006,tolerant2025} considers the composite problem $d_{TV}(P,Q)\leqslant \varepsilon_1$ versus $d_{TV}(P,Q)> \varepsilon_2$, with gap $\Delta=\varepsilon_2-\varepsilon_1$. Its optimal sample complexity is now well understood~\cite{price2021tolerant}, implying that the minimal detectable gap satisfies $\Delta \gtrsim d^{1/4}n^{-1/2}$ in the classical regime and $\Delta \gtrsim \sqrt{d/n}$ in the worst case;  moreover, tolerant testing is intrinsically harder, with at most a quadratic gap from the non-tolerant setting~\cite{chakraborty2022gap}. These results are tailored to discrete domains and to the total-variation distance.

Our work subsumes total-variation tolerant testing as a special case, but takes a different perspective. Rather than focusing on one discrepancy on a discrete domain, we study testing of an arbitrary benchmark trade-off curve. This yields a common framework that applies to both discrete and continuous models and covers privacy-style benchmarks as well as total variation. In the discrete case, taking $\mathcal S=2^{[d]}$ gives $\VC(\mathcal S)=d$, and our sufficient condition recovers the worst-case tolerant scaling $\Delta \gtrsim \widetilde O(\sqrt{d/n})$. We do not recover the sharper $d^{1/4}n^{-1/2}$ scaling available in the special noiseless discrete setting, which is consistent with the fact that our results are designed as a general structural theory for trade-off testing rather than a minimax analysis specialized to discrete total-variation testing.

\section{Discussion}\label{sec:discussion}

This paper identifies the structural boundary between possibility and
impossibility for testing trade-off functions from finite samples. 
The answer
rests on two conditions---\(\mathcal{S}\)-attainability and finite VC
dimension---which together convert an infinite-dimensional inference problem
into one governed by uniform convergence over a structured set class. The
resulting test enjoys assumption-free type~I error control, achieves
minimax-optimal separation rates up to logarithmic factors for the MLR model class, and degrades
gracefully under model misspecification. We close by highlighting several
directions for future work.

\paragraph{Minimax rates for general \(\mathcal{S}\) classes.}
Our minimax lower bound (\Cref{thm:new-lower-bound}) is established for
half-intervals, the simplest nontrivial class with \(\VC(\mathcal{S})=2\).
Since the minimax risk over \(\cM_{\mathcal{S}}\) can only increase as
\(\mathcal{S}\) grows, this lower bound applies to any class containing
half-intervals. However, the upper bound depends on \(\VC(\mathcal{S})\)
through the parameter \(\tauND\), so when \(\VC(\mathcal{S})\) is large there
remains an implicit gap of order
$
\sqrt{\VC(\mathcal{S})\log n}
$
between the upper and lower bounds. A natural question is whether the
minimax-optimal separation rate for a general class \(\mathcal{S}\) depends on
\(\VC(\mathcal{S})\) only through \(\tauND\).

\paragraph{Choosing \(\mathcal{S}\) in practice.}
Our framework treats the set class \(\mathcal{S}\) as given, and the
theoretical results show that this choice mediates a bias--variance trade-off:
a richer class reduces the approximation error due to misspecification, but it
also inflates the estimation error through a larger VC dimension. In practice,
the choice of \(\mathcal{S}\) must therefore be guided by domain knowledge
about the mechanism generating \((P,Q)\).

\paragraph{Computation.}
We show that the test reduces to a
finite collection of cost-sensitive empirical risk minimization problems over
\(\mathcal S\). This makes the method efficient for several natural classes,
such as halfspaces in fixed dimension and unions of intervals on the line. For
more complex classes, however, the corresponding ERM problem may be
computationally difficult even when the class has finite VC dimension.
Understanding the statistical-computational trade-offs in trade-off testing,
and developing approximate optimization methods with valid error control, would
broaden the applicability of the framework.

\bibliography{citation}
\bibliographystyle{alpha}

\newpage
\appendix

\section{Preliminaries}

\subsection{Neyman--Pearson}
\begin{lemma}[Neyman--Pearson, general form]\label{lem:neyman-pearson}
Let $P,Q$ be probability measures on a measurable space $(\Omega,\mathcal F)$, and set
\[
\mu=\tfrac{1}{2}(P+Q),\qquad p=\frac{dP}{d\mu},\qquad q=\frac{dQ}{d\mu},
\]
so that $p,q\in[0,2]$ and $p+q=2$ $\mu$-a.e. For $\lambda\in[0,\infty]$, define
\[
S_{>\lambda}:=\{q>\lambda p\},\qquad S_{\geqslant\lambda}:=\{q\geqslant\lambda p\},
\]
with the conventions $S_{>\infty}:=\{p=0,\,q>0\}$ and $S_{\geqslant\infty}:=\{p=0\}$.
A \emph{test} is a measurable function $\phi:\Omega\to[0,1]$, with size $E_P[\phi]=\int\phi\,p\,d\mu$ and power $E_Q[\phi]=\int\phi\,q\,d\mu$.

For every $\alpha\in[0,1]$, the following hold.

\begin{enumerate}
\item[(a)] \emph{(Existence.)} There exist $\lambda\in[0,\infty]$ and $\gamma\in[0,1]$ such that the test
\[
\phi^\star=\mathbf 1_{S_{>\lambda}}+\gamma\bigl(\mathbf 1_{S_{\geqslant\lambda}}-\mathbf 1_{S_{>\lambda}}\bigr)
\]
satisfies $E_P[\phi^\star]=\alpha$.

\item[(b)] \emph{(Sufficiency.)} Any test $\phi^\star$ of the form in (a) is most powerful at its own size: for every test $\phi$,
\[
E_P[\phi]\leqslant E_P[\phi^\star]\quad\Longrightarrow\quad E_Q[\phi]\leqslant E_Q[\phi^\star].
\]

\item[(c)] \emph{(Necessity.)} If $\phi$ is most powerful at level $\alpha$ with $E_P[\phi]=\alpha$, then there exists $\lambda\in[0,\infty]$ such that
\[
\phi=\mathbf 1_{S_{>\lambda}}\quad\mu\text{-a.e. on } \{q\neq\lambda p\}.
\]
\end{enumerate}
\end{lemma}

\begin{lemma}[Witness Property]\label{lem:witness}
Let $(P,Q)$ be a pair of distributions in $\mathcal M_{\mathcal S}$ and $f$ be a convex, continuous, and non-increasing function. Suppose that $\Toff{P}{Q}(\alpha_0) < f(\alpha_0)$ for some $\alpha_0 \in [0,1]$ , then there exists $S^\star\in \mathcal S$ such that
\[
Q\bigl((S^\star)^\complement\bigr) < f\bigl(P(S^\star)\bigr).
\]
\end{lemma}

\begin{proof}
Let $\mu=(P+Q)/2$, $p=dP/d\mu$, and $q=dQ/d\mu$.
By the Neyman--Pearson lemma, there exist $\lambda \in [0, \infty]$ and $\gamma\in[0,1]$ such that the optimal level-$\alpha_0$ test is
\[
\psi=(1-\gamma)\mathbf 1_{S_{>\lambda}}+\gamma \mathbf 1_{S_{\geqslant \lambda}},
\]
where
\[
S_{>\lambda}:=\{q>\lambda p\},
\qquad
S_{\geqslant\lambda}:=\{q\geqslant\lambda p\}.
\]
Its type~I and type~II errors decompose as
\begin{equation}\label{eq:witness-decomp}
\bigl(\alpha_0,\,\Toff{P}{Q}(\alpha_0)\bigr)
=
(1-\gamma)\bigl(P(S_{>\lambda}),\,Q(S_{>\lambda}^\complement)\bigr)
+
\gamma\bigl(P(S_{\geqslant\lambda}),\,Q(S_{\geqslant\lambda}^\complement)\bigr).
\end{equation}

Convexity of $f$ gives
\[
f(\alpha_0)\leqslant (1-\gamma)f(P(S_{>\lambda}))+\gamma f(P(S_{\geqslant\lambda}));
\]
combined with $\Toff{P}{Q}(\alpha_0)<f(\alpha_0)$ and \eqref{eq:witness-decomp}, at least one of the two vertex inequalities must be strict:
\begin{equation}\label{eq:witness-dichotomy}
Q(S_{>\lambda}^\complement)<f\bigl(P(S_{>\lambda})\bigr)
\quad\text{or}\quad
Q(S_{\geqslant\lambda}^\complement)<f\bigl(P(S_{\geqslant\lambda})\bigr).
\end{equation}

\emph{Case A.} If the first inequality in \eqref{eq:witness-dichotomy} holds, then $S_{>\lambda}\in\cS$ by $\cS$-attainability and the conclusion holds with $S^\star=S_{>\lambda}$.

\emph{Case B.} Suppose only the second inequality in \eqref{eq:witness-dichotomy} holds.
If $\lambda=0$, then $S_{\geqslant\lambda}=\mathcal X\in\cS$, so the conclusion holds with $S^\star=\mathcal X$.
Now assume $\lambda>0$. Take any sequence $\lambda_n$ that increases to $\lambda$; set
\[
S_{>\lambda_n}:=\{q>(\lambda_n)p\}\in\cS;
\]
these sets decrease to $S_{\geqslant\lambda}$ as $\lambda_n \to \lambda$, so by continuity of probability and of $f$, the strict inequality at $S_{\geqslant\lambda}$ persists at $S_{>\lambda-\varepsilon}$ for all sufficiently small $\varepsilon>0$. The conclusion then holds with $S^\star=S_{>\lambda-\varepsilon}$.
\end{proof}

\begin{lemma}[Convex-hull envelope equals the greatest convex minorant]
\label{lem:convex-hull-gcm}
Let $S$ be a class of measurable sets containing $\varnothing$ and $X$, and define
\[
C:=\{(P(A),\,Q(A^{\complement})):\ A\in S\}\subset [0,1]^2.
\]
For $\alpha\in[0,1]$, define
\[
F(\alpha):=\inf_{A\in S:\,P(A)\leqslant \alpha} Q(A^{\complement}),
\qquad
G(\alpha):=\inf\{\beta:\ (\alpha,\beta)\in \operatorname{conv}(C)\}.
\]
Then $G=\operatorname{GCM}(F)$, 
where $\operatorname{GCM}(F)$ denotes the greatest convex minorant of $F$ on $[0,1]$.
\end{lemma}

\begin{proof}
We first show that $G\leqslant F$ pointwise. Fix $\alpha\in[0,1]$ and $A\in S$ with
\[
x:=P(A)\leqslant \alpha,
\qquad
y:=Q(A^{\complement}).
\]
If $x=\alpha$, then $(\alpha,y)\in C\subset \operatorname{conv}(C)$, hence $G(\alpha)\leqslant y$.

If $x<\alpha$, use that $X\in S$, so $(1,0)\in C$. Set
\[
\theta:=\frac{1-\alpha}{1-x}\in[0,1].
\]
Then
\[
\theta(x,y)+(1-\theta)(1,0)=(\alpha,\theta y)\in \operatorname{conv}(C),
\]
so
\[
G(\alpha)\leqslant \theta y\leqslant y.
\]
Taking the infimum over all $A\in S$ with $P(A)\leqslant \alpha$ yields
\[
G(\alpha)\leqslant F(\alpha).
\]

Next we show that $G$ dominates every convex minorant of $F$. Let $g$ be any convex function on
$[0,1]$ such that $g\leqslant F$. For every $A\in S$,
\[
g(P(A))\leqslant F(P(A))\leqslant Q(A^{\complement}),
\]
since $A$ itself is admissible in the definition of $F(P(A))$. Hence
\[
C\subset \operatorname{epi}(g):=\{(u,v)\in[0,1]^2:\ v\geqslant g(u)\}.
\]
Because $g$ is convex, $\operatorname{epi}(g)$ is convex. Therefore
\[
\operatorname{conv}(C)\subset \operatorname{epi}(g).
\]
So for every $(\alpha,\beta)\in \operatorname{conv}(C)$,
\[
\beta\geqslant g(\alpha).
\]
Taking the infimum over such $\beta$ gives
\[
G(\alpha)\geqslant g(\alpha)
\qquad\forall \alpha\in[0,1].
\]

Thus $G$ is itself a convex minorant of $F$, and it dominates every convex minorant of $F$.
Therefore $G=\operatorname{GCM}(F)$.
\end{proof}

\subsection{Properties of $\hup$}
Simple calculations yield the explicit formula of $\hinv$:
\begin{align*}
    \hinv(t)=\begin{dcases}
        0,\quad &0\leqslant t\leqslant \tau_{n,\delta};\\
        t-\frac{\tau_{n,\delta}+\sqrt{\tau_{n,\delta}(4t-3\tau_{n,\delta})}}{2},\quad &\tau_{n,\delta}\leqslant t\leqslant 1\wedge(1/2+\sqrt{\tau_{n,\delta}/2}+\tau_{n,\delta});\\
        t-\frac{3\tau_{n,\delta}+\sqrt{\tau_{n,\delta}(4+5\tau_{n,\delta}-4t)}}{2},\quad &1\wedge(1/2+\sqrt{\tau_{n,\delta}/2}+\tau_{n,\delta})\leqslant t\leqslant 1.
        \end{dcases}
    \end{align*}
The following lemma collects the properties of $\hinv$ used throughout the appendix.

\begin{lemma}\label{lem:hinv-properties}
Write $A_t:=\{s\in[0,1]:h_+(s)\geqslant t\}$ for the super-level set of $h_+$. Then:
\begin{enumerate}[label=(\roman*)]
    \item for every $t\in[0,1]$, $A_t=[\hinv(t),1]$;
    \item $\hinv$ is convex and non-decreasing on $[0,1]$;
    \item for all $s,t\in[0,1]$, $h_+(s)<t\iff s<\hinv(t)$. In particular, since the ranges of $Q_n(S^\complement)$ and $\fnull(\cdot)$ lie in $[0,1]$,
    \[
        h_+(Q_n(S^\complement))<\fnull\circ h_+(P_n(S))\iff Q_n(S^\complement)<\hinv\circ\fnull\circ h_+(P_n(S)).
    \]
\end{enumerate}
\end{lemma}

\begin{proof}
We first note that $h_+$ is concave on $[0,1]$: it is continuous, with second derivative $h_+''(t)=-\sqrt{\tau}/t^{3/2}<0$ on $(0,1/2)$ and $h_+''(t)=-\sqrt{\tau}/(1-t)^{3/2}<0$ on $(1/2,1)$, and left/right derivatives at $1/2$ satisfy $h_+'(1/2^-)=1+\sqrt{\tau}/2>1-\sqrt{\tau}/2=h_+'(1/2^+)$.

\emph{Proof of (i).} Since $h_+(1)=1\geqslant t$, we have $1\in A_t$. By definition, $\inf A_t=\hinv(t)$. Since $h_+$ is concave, $A_t$ is a convex closed set, i.e., a closed interval, so $A_t=[\hinv(t),1]$.

\emph{Proof of (ii).} That $\hinv$ is non-decreasing follows from the definition of left-continuous inverse. For convexity, fix $\lambda\in[0,1]$ and $t_1,t_2\in[0,1]$. By (i), $h_+(\hinv(t_i))\geqslant t_i$ for $i=1,2$, so by concavity of $h_+$,
\[
    h_+\bigl(\lambda\hinv(t_1)+(1-\lambda)\hinv(t_2)\bigr)
    \geqslant \lambda h_+(\hinv(t_1))+(1-\lambda)h_+(\hinv(t_2))
    \geqslant \lambda t_1+(1-\lambda)t_2.
\]
Hence $\lambda\hinv(t_1)+(1-\lambda)\hinv(t_2)\in A_{\lambda t_1+(1-\lambda)t_2}$, so $\lambda\hinv(t_1)+(1-\lambda)\hinv(t_2)\geqslant \inf A_{\lambda t_1+(1-\lambda)t_2}=\hinv(\lambda t_1+(1-\lambda)t_2)$.

\emph{Proof of (iii).} Suppose $h_+(s)<t$. If $s\geqslant\hinv(t)$, then $s\in A_t$ by (i), giving $h_+(s)\geqslant t$---contradiction. Conversely, if $s<\hinv(t)$, then $s\notin A_t$ by (i), so $h_+(s)<t$. Specializing to $s=Q_n(S^\complement)$ and $t=\fnull\circ h_+(P_n(S))$ yields the second equivalence.
\end{proof}

Recall that 
\[
\widetilde{\fnull} = \hinv\circ\fnull\circ\hup.
\]
We have the following property of $\widetilde{\fnull}$.

\begin{lemma}\label{lem:tilde-f}
    The function $\widetilde{\fnull} : [0,\infty) \mapsto [0,1]$ is convex and non-increasing. 
\end{lemma}

\begin{proof}
    For any $\lambda\in[0,1]$ and $t_1,t_2 \geqslant 0$, since $\hup$ is concave, we have
\[
\hup ( \lambda t_1+(1-\lambda) t_2 ) \geqslant \lambda\hup(t_1)+(1-\lambda)\hup(t_2).
\]
In addition, we know that $\fnull$ is non-increasing and $\hinv$ is non-decreasing, the composition~$\hinv \circ \fnull$ is non-increasing. 
As a result, we have 
\begin{align*}
\hinv\circ\fnull\circ\hup(\lambda t_1+(1-\lambda) t_2)
&\leqslant \hinv\circ\fnull\bigl(\lambda\hup(t_1)+(1-\lambda)\hup(t_2)\bigr)\\
&\leqslant \hinv\bigl(\lambda\,\fnull\circ\hup(t_1)+(1-\lambda)\,\fnull\circ\hup(t_2)\bigr)\\
&\leqslant \lambda\,\hinv\circ\fnull\circ\hup(t_1)+(1-\lambda)\,\hinv\circ\fnull\circ\hup(t_2),
\end{align*}
Here the middle inequality uses the convexity of $\fnull$, and last one uses the convexity of $\hinv$. 
In all, this establishes the convexity of $\widetilde{\fnull}$.

It remains to prove that $\widetilde{\fnull}$ is non-increasing, which is an easy consequence of (1) $\hup$ is non-decreasing, $\fnull$ is non-increasing, and $\hinv$ is non-decreasing.
\end{proof}

\section{Proof of \Cref{thm:mlr}}
\label{sec:proof-mlr}

Apply \Cref{lemma:DKW} to \(X_{1:n}\sim P\) and \(Y_{1:n}\sim Q\) separately,
each with failure probability \(\delta/2\). By a union bound, with probability at
least \(1-\delta\),
\begin{equation}
\label{eq:good-event-mlr}
\tag{\(\mathcal E\)}
\sup_{S\in\mathcal S_{\mathrm{MLR}}}| {P_n(S)-P(S)} |\leqslant \etaND,
\qquad
\sup_{S\in\mathcal S_{\mathrm{MLR}}}| {Q_n(S)-Q(S)} | \leqslant \etaND,
\end{equation}
where
$
\etaND
=
\sqrt{\frac{1}{2n}\log\frac{4}{\delta}}.
$
We prove both claims on this good event \(\mathcal E\).

\paragraph{Type~I error control.}
Assume $\Toff{P}{Q}\succeq \fnull$. Then for every $S\in\smlr$,
\[
Q_n(S^\complement)+\etaND
\geqslant Q(S^\complement)
\geqslant \fnull\bigl(P(S)\bigr)
\geqslant \fnull\bigl(P_n(S)+\etaND\bigr),
\]
where the first and third inequalities use $\mathcal E$, the second uses the null hypothesis together with the definition of the trade-off function, and the third also uses that $\fnull$ is non-increasing. The rejection condition thus never occurs on $\mathcal E$, giving $\psi_{n,\delta}=0$ on $\mathcal E$ and hence $\errone(\psi_{n,\delta})\leqslant\delta$.

\paragraph{Type~II error control.}
Assume \((P,Q)\in\mathcal M_{\mathrm{MLR}}\) and \(\Toff{P}{Q}\nsucceq \falt\).
Since every continuous MLR pair is \(\smlr\)-attainable, we have
\[
(P,Q)\in\mathcal M_{\smlr}.
\]
Therefore \Cref{lem:witness} yields a set \(S^\star\in\smlr\) such that
\begin{equation}\label{eq:witness-properties}
0\leqslant Q\bigl((S^\star)^\complement\bigr)<\falt\bigl(P(S^\star)\bigr).
\end{equation}
The separation condition~\eqref{eq:mlr-sep} at $\alpha=P(S^\star)$ gives $\falt(P(S^\star))\leqslant\max\{\fnull(P(S^\star)+2\etaND)-2\etaND,\,0\}$; since $\falt(P(S^\star))>0$ by \eqref{eq:witness-properties}, the maximum is attained by the first branch, yielding
\begin{equation}\label{eq:sep-rewritten}
\falt\bigl(P(S^\star)\bigr)+2\etaND\leqslant \fnull\bigl(P(S^\star)+2\etaND\bigr).
\end{equation}
On $\mathcal E$, combining \eqref{eq:witness-properties}, \eqref{eq:sep-rewritten}, $P_n(S^\star)\leqslant P(S^\star)+\etaND$, and monotonicity of $\fnull$, we have
\[
Q_n\bigl((S^\star)^\complement\bigr)+\etaND
\leqslant Q\bigl((S^\star)^\complement\bigr)+2\etaND
<\falt\bigl(P(S^\star)\bigr)+2\etaND
\leqslant \fnull\bigl(P(S^\star)+2\etaND\bigr)
\leqslant \fnull\bigl(P_n(S^\star)+\etaND\bigr).
\]
Thus the rejection condition holds for $S^\star$, so $\psi_{n,\delta}=1$ on $\mathcal E$ and $\errtwo(\cM_{\smlr};\psi_{n,\delta})\leqslant\delta$.

\section{Proofs for \Cref{sec:general}}\label{app:general}

\subsection{Proof of \Cref{prop:reduction}}\label{sec:proof-convex-closure}

Let
\[
\cC := \{(P(S),\,Q(S^\complement)) : S\in\cS\}.
\]
We show that for every $\alpha\in[0,1]$,
\[
\Toff{P}{Q}(\alpha)
=
\inf\{\beta : (\alpha,\beta)\in \convclosure(\cC)\}.
\]
The proof is split into two inequalities.

\paragraph{Upper bound.}
Pick $(\alpha,\beta)\in\mathrm{conv}(\cC)$. By definition of the convex hull, there exist $k\geqslant 1$, convex weights $\lambda_1,\dots,\lambda_k\geqslant 0$ with $\sum_i\lambda_i=1$, and sets $S_1,\dots,S_k\in\cS$ with $(\alpha,\beta)=\sum_i\lambda_i(P(S_i),Q(S_i^\complement))$. The randomized test $\varphi:=\sum_i\lambda_i\mathbf 1_{S_i}$ takes values in $[0,1]$ and satisfies $E_P[\varphi]=\alpha$, $E_Q[1-\varphi]=\beta$, so $\Toff{P}{Q}(\alpha)\leqslant\beta$. Since this holds for every such $(\alpha,\beta)$ and $\Toff{P}{Q}$ is continuous, it also holds on $\convclosure(\cC)$, giving
\[
\Toff{P}{Q}(\alpha)\leqslant \inf\bigl\{\beta:(\alpha,\beta)\in\convclosure(\cC)\bigr\}.
\]

\paragraph{Lower bound.}
Let $\mu=(P+Q)/2$, $p=dP/d\mu$, and $q=dQ/d\mu$.
By the Neyman--Pearson lemma, there exist $r \geqslant 0$ and $\gamma \in [0,1]$ such that the test
\[
\varphi^\star(x) = \begin{cases}
1, & \text{if } q(x) > r\,p(x), \\[2pt]
\gamma, & \text{if } q(x) = r\,p(x), \\[2pt]
0, & \text{if } q(x) < r\,p(x),
\end{cases}
\]
satisfies
\[
E_P[\varphi^\star] = \alpha
\qquad\text{and}\qquad
E_Q[1-\varphi^\star] = \Toff{P}{Q}(\alpha).
\]
Writing
\[
S_{>}:=\{q>rp\},
\qquad
S_{\geqslant}:=\{q\geqslant rp\},
\]
we have
\[
\varphi^\star=(1-\gamma)\mathbf{1}_{S_{>}}+\gamma\mathbf{1}_{S_{\geqslant}},
\]
and therefore
\begin{equation}\label{eq:convex-decomp}
\bigl(\alpha,\; \Toff{P}{Q}(\alpha)\bigr)
=
(1-\gamma)\,\underbrace{\bigl(P(S_{>}),\, Q(S_{>}^\complement)\bigr)}_{\text{point } A}
\;+\;
\gamma\,\underbrace{\bigl(P(S_{\geqslant}),\, Q(S_{\geqslant}^\complement)\bigr)}_{\text{point } B}.
\end{equation}
By $\cS$-attainability, $S_{>}\in\cS$ (up to a $(P+Q)/2$-null set), so $A\in\cC$. The non-strict level set $S_{\geqslant}$ need not itself lie in $\cS$; we handle it by approximation.

If $r>0$, then for $\varepsilon\in(0,r)$ the set $S_{>}^{(\varepsilon)}:=\{dQ/dP>r-\varepsilon\}$ lies in $\cS$ by $\cS$-attainability and decreases to $S_{\geqslant}$ as $\varepsilon\downarrow 0$; continuity of probability yields $(P(S_{\geqslant}),Q(S_{\geqslant}^\complement))=\lim_{\varepsilon\to 0^+}(P(S_{>}^{(\varepsilon)}),Q((S_{>}^{(\varepsilon)})^\complement))$, so $B\in\overline{\cC}$. If $r=0$, then $S_{\geqslant}=\cX\in\cS$ by assumption and $B=(1,0)\in\cC$ directly. In either case, from~\eqref{eq:convex-decomp} we obtain $(\alpha, \Toff{P}{Q}(\alpha)) \in \convclosure(\mathcal{C})$, and therefore
\[
\Toff{P}{Q}(\alpha) \geqslant \inf\bigl\{\beta : (\alpha,\beta) \in \convclosure(\mathcal{C})\bigr\}.
\]
Combining with the upper bound completes the proof.

\subsection{Proof of \Cref{thm:vc-necessary}}\label{sec:VC-necessity}

We prove that if $\VC(\mathcal S)=\infty$, then no test can uniformly distinguish
\[
H_0=\{(P,Q)\in\cM_{\cS}: \Toff{P}{Q}\succeq \fnull\}
\qquad\text{from}\qquad
H_1=\{(P,Q)\in\cM_{\cS}: \Toff{P}{Q}\nsucceq \falt\}.
\]

Fix a positive integer $m$; we construct a null pair $(P,P)$ and alternative pairs $(P,Q_A)$ indexed by subsets $A\subseteq[m]$ of fixed size, then bound the testing error via a chi-squared calculation.

\paragraph{Step 1: construction.}
By the continuity of the trade-off function $\falt$ and $\falt(0) > 0$, we can pick some $\alpha_0\in(0,1)$ so that $\falt(\alpha_0)>0$. 
Consider any fixed positive integer $m$ obeying $m>1/\alpha_0$.
Since $\VC(\cS)=+\infty$, for any $m$, there exist mutually distinct points $\{x_i\}_{1 \leqslant i \leqslant m}\subset \mathbb{R}^d$ (or $\bbZ$) that can be shattered by $\cS$. Therefore, for any $A\subseteq [m]$, there exists $S_A\in\cS$ such that
\begin{align*}
    \{x_i:\ i\in[m]\}\bigcap S_A=\{x_i:\ i\in A\}.
\end{align*}

Now, let 
\begin{align*}
    P=\frac{1}{m}\sum_{i=1}^m \delta_{x_i}.
\end{align*}
Pick some integer $a\in[m]$, and define $Q_A:\, A\subseteq [m],\ \abs{A}=a$ by
\begin{align*}
    \dif Q_A(x)=\begin{dcases}
        \beta\cdot \dif P(x),\quad & x\in \cup_{i\in A} \{x_i\},\\
        \gamma\cdot \dif P(x),\quad & x\in \cup_{i\notin A} \{x_i\},
    \end{dcases}
\end{align*}
where $\beta$ and $\gamma$ are positive values to be determined later that satisfy
\begin{align*}
    \beta a+\gamma (m-a)=m.
\end{align*}
Note that
\begin{align*}
    \frac{\dif Q_A}{\dif P}(x)=\begin{dcases}
        \beta,\quad &x\in \cup_{i\in A} \{x_i\},\\
        \gamma,\quad &x\in \cup_{i\notin A} \{x_i\},\\
        0/0,\quad &\textnormal{otherwise}.
    \end{dcases}
\end{align*}
It is straightforward that every pair $(P,Q_A)$ is $\cS$-attainable since $\cup_{i\in A}\{x_i\}=S_A\in \cS$ in the $(P+Q_A)/2$-a.e. sense for any $A\subseteq [m]$.

From now on, we set $\beta=(1-\falt(\alpha_0))/(\alpha_0-1/m)$,  and $\abs{A}=\floor{\alpha_0 m}$.

\paragraph{Step 2: Null and alternative membership.}

Clearly, we have $(P,P)$ is within $H_0$. 
In this step, we aim to show that $(P, Q_A)$ lies in the alternative $H_1$.
To see this, consider the test $\phi_A=\indi\{X\in S_A\}$, which obeys
\begin{align*}
    P\phi_A&=P(S_A)=\abs{A}/m= \abs{A}/m, \\ Q(1-\phi_A)&=Q(S_A^\complement)=\gamma (1-\abs{A}/m)=1-\beta\abs{A}/m.
\end{align*}
Recall that $\abs{A} = \floor{\alpha_0 m}$ and $\beta=(1-\falt(\alpha_0))/(\alpha_0-1/m)$.
We have
\begin{align*}
    P\phi_A&=\floor{\alpha_0 m}/m\leqslant \alpha_0,\\
    Q(1-\phi_A)&=1-(1 - \falt(\alpha_0)) \cdot\frac{\floor{\alpha_0 m}}{ \alpha_0 m -1} < \falt(\alpha_0),
\end{align*}
where the strict inequality follows from the choice $m > 1/\alpha_0$.

In all, the test $\phi_A=\indi\{X\in S_A\}$ witnesses the fact that 
$\Toff{P}{Q_A}(\alpha_0)< \falt(\alpha_0)$, and hence $(P,Q_A)$ is in the alternative.

\paragraph{Step 3: chi-squared bound.}
We now specify that $a = \floor{\alpha_0 m}$. Let $\Pi$ denote the uniform distribution over
$\{A \subseteq [m] : |A| = a\}$. Since $(P,P)$ is in the null
and every $(P, Q_A)$ is in the alternative,
\begin{align}
\label{eq:vc-testing-lb}
\inf_\psi \Bigl\{\errone(\cM_\cS;\psi)
+ \errtwo(\cM_\cS;\psi)\Bigr\}
&\geqslant \inf_\psi \Bigl\{
P^{2n}( \psi)
+ \E_{A \sim \Pi}\bigl[(P^n \times Q_A^n)(1-\psi)\bigr]
\Bigr\} \nonumber\\
&= 1 - \TV\!\Bigl(P^{2n},
\E_{A \sim \Pi}[P^n \times Q_A^n]\Bigr) \nonumber\\
&\geqslant 1 - \sqrt{\tfrac12\,
\chi^2\!\Bigl(\E_{A \sim \Pi}[P^n \times Q_A^n] \,\big\|\, P^{2n}\Bigr)},
\end{align}
where the first line uses that the infimum over all tests of
$\mu_0(1-\psi) + \mu_1(\psi)$ equals $1 - \TV(\mu_0,\mu_1)$, and
the last line is the standard TV--$\chi^2$ bound.

For the chi-squared divergence, independence gives
\begin{align*}
1 + \chi^2\!\Bigl(\E_\Pi[P^n \times Q_A^n] \,\big\|\, P^{2n}\Bigr)
&= \E_{A,A' \sim \Pi}\Biggl[
\Biggl(\E_{X \sim P}\Bigl[
\frac{dQ_A}{dP}(X)\,\frac{dQ_{A'}}{dP}(X)
\Bigr]\Biggr)^{n}
\Biggr].
\end{align*}
Use the definitions of $Q_A, Q_{A'}$ and $P$ to obtain
\[
\E_{X \sim P}\Bigl[
\frac{dQ_A}{dP}(X)\,\frac{dQ_{A'}}{dP}(X)\Bigr]
= 1 + (\beta - \gamma)^2
\Bigl(\frac{|A \cap A'|}{m} - \bigl(\tfrac{a}{m}\bigr)^2\Bigr).
\]
Using $1 + u \leqslant e^u$ and taking expectations,
\begin{align*}
1 + \chi^2
&\leqslant \E_{A,A' \sim \Pi}\biggl[
\exp\biggl(
n(\beta - \gamma)^2\Bigl(
\frac{|A \cap A'|}{m} - \bigl(\tfrac{a}{m}\bigr)^2
\Bigr)
\biggr)\biggr].
\end{align*}
Under $\Pi$, $|A\cap A'|\sim\HyperGeom(m,a,a)$, whose MGF is dominated by that of $Z\sim\Binom(a,a/m)$ (by the convex ordering of hypergeometric and binomial with the same mean). Substituting the binomial MGF and applying $1+x\leqslant e^x$ and $e^t-t-1\leqslant t^2e^t$ in sequence,
\begin{align*}
1 + \chi^2
&\leqslant \E\!\left[\exp\!\left(n(\beta-\gamma)^2\Bigl(\tfrac{Z}{m}-\bigl(\tfrac{a}{m}\bigr)^2\Bigr)\right)\right]
= \exp\!\Bigl(-\tfrac{n(\beta-\gamma)^2 a^2}{m^2}\Bigr)\Bigl(1-\tfrac{a}{m}+\tfrac{a}{m}\,e^{(\beta-\gamma)^2 n/m}\Bigr)^{a}\\
&\leqslant \exp\!\Bigl(\tfrac{a^2}{m}\bigl(e^{(\beta-\gamma)^2 n/m}-(\beta-\gamma)^2 n/m-1\bigr)\Bigr)
\leqslant \exp\!\Bigl(\tfrac{a^2(\beta-\gamma)^4n^2}{m^3}\,e^{(\beta-\gamma)^2n/m}\Bigr).
\end{align*}
Since $\beta-\gamma=O_{\falt}(1)$ is bounded independently of $m$, the right-hand side tends to $1$ as $m\to\infty$, i.e., $\chi^2\to 0$. Together with~\eqref{eq:vc-testing-lb}, this completes the proof.

\subsection{Proof of \Cref{thm:general}}
\label{sec:proof-general}

Since \(\VC(\cS^\complement)=\VC(\cS)\), \Cref{lem:nuc} applies to both
\(\cS\) and \(\cS^\complement\). By a union bound, with probability at least
\(1-\delta\), the event
\[
\mathcal E
=
\Bigl\{
\forall S\in \cS\cup\cS^\complement:\ 
P(S)\leqslant \hup(P_n(S)),\ 
P_n(S)\leqslant \hup(P(S)),\ 
Q(S)\leqslant \hup(Q_n(S)),\ 
Q_n(S)\leqslant \hup(Q(S))
\Bigr\}
\]
holds.
We prove both claims on this good event \(\mathcal E\).

\paragraph{Type~I error control.}
Assume \(\Toff{P}{Q}\succeq \fnull\). For each \(S\in\cS\),
\[
\hup(Q_n(S^\complement))
\geqslant Q(S^\complement)
\geqslant \fnull(P(S))
\geqslant \fnull\bigl(\hup(P_n(S))\bigr),
\]
where the first and third inequalities use \(\mathcal E\), and the last one uses
that \(\fnull\) is nonincreasing. Therefore the rejection condition in
\eqref{eq:general-test} never occurs on \(\mathcal E\). Hence
$\errone(\psi_{n,\delta})\leqslant \delta$.

\paragraph{Type~II error control.}
Suppose $(P,Q)\in\cM_{\cS}$ and $\Toff{P}{Q}\nsucceq\falt$. 
Then there exists some $\alpha_0 \in [0,1]$ such that $\Toff{P}{Q}(\alpha_0) < \falt(\alpha_0)$.
\Cref{lem:witness} then yields $S^\star\in\cS$ with $Q((S^\star)^\complement)<\falt(P(S^\star))$.

The map $\hinv\circ\widetilde{\fnull}$ is non-increasing, since $\widetilde{\fnull}$ is non-increasing (\Cref{prop:equiv-surrogate-convex}) and $\hinv$ is non-decreasing (\Cref{lem:hinv-properties}(ii)).  
The separation assumption~\eqref{eq:general-sep}, the identity $\hinvinv\circ\fnull\circ\hupup=\hinv\circ\widetilde{\fnull}\circ\hup$, and the bound $\hup(P(S^\star))\geqslant P_n(S^\star)$ on $\mathcal E$ combine to give
\[
Q((S^\star)^\complement)
<\falt(P(S^\star))
\leqslant \hinvinv\circ\fnull\circ\hupup(P(S^\star))
=\hinv\circ\widetilde{\fnull}\circ\hup(P(S^\star))
\leqslant \hinv\circ\widetilde{\fnull}(P_n(S^\star)).
\]
By \Cref{lem:hinv-properties}(iii), this is equivalent to $\hup(Q((S^\star)^\complement))<\widetilde{\fnull}(P_n(S^\star))$. On $\mathcal E$ we have $Q_n((S^\star)^\complement)\leqslant\hup(Q((S^\star)^\complement))$, so $Q_n((S^\star)^\complement)<\widetilde{\fnull}(P_n(S^\star))$. \Cref{prop:equiv-surrogate-convex} then shows that the test~\eqref{eq:general-test} rejects on $\mathcal E$, giving $\errtwo(\cM_{\cS};\psi_{n,\delta})\leqslant\delta$.

\subsection{Proof of \Cref{cor:confidence-bands}}\label{sec:proof-confidence-bands}

By \Cref{lem:nuc}, with probability at least $1-\delta$ the event
\[
\mathcal E=\Bigl\{\forall S\in\cS\cup\cS^\complement,\ \mu\in\{P,Q\}:\ \hdown(\mu_n(S))\leqslant\mu(S)\leqslant\hup(\mu_n(S))\ \text{and}\ \hdown(\mu(S))\leqslant\mu_n(S)\leqslant\hup(\mu(S))\Bigr\}
\]
holds; we work on $\mathcal E$.

\paragraph{Confidence upper bound.}
For any $\alpha\in[0,1]$, by definition of the trade-off function, on the event $\cE$, we have 
\[
\Toff{P}{Q}(\alpha)\leqslant \inf_{S\in\cS:\,P(S)\leqslant\alpha}Q(S^\complement)
\leqslant \inf_{S\in\cS:\,\hup(P_n(S))\leqslant\alpha}\hup(Q_n(S^\complement)).
\]
Since trade-off functions are bounded by $1-\alpha$, this yields the claimed upper bound.

\paragraph{Confidence lower bound.}
Assume for now that $(P,Q) \in \cM_\cS$. Define $F(\alpha):=\inf_{S\in\cS:\,P(S)\leqslant\alpha}Q(S^\complement)$. 

\Cref{lem:convex-hull-gcm} gives
\[
T(P,Q)=\operatorname{GCM}(F).
\] On $\mathcal E$,
\[
F(\alpha)\geqslant
\left[
\inf_{A\in S:\,h_-(P_n(A))\leqslant \alpha} h_-(Q_n(A^{\complement}))
\right]\vee 0
=
\widehat L(\alpha).
\]
Thus $\widehat L_{\mathrm{GCM}}$ is a convex minorant of $F$. By maximality of the greatest convex
minorant,
\[
\widehat L_{\mathrm{GCM}}(\alpha)\leqslant \operatorname{GCM}(F)(\alpha)=T(P,Q)(\alpha)
\qquad\forall \alpha\in[0,1].
\]

\subsection{Proof of \Cref{cor:adaptive}}

Under $H_0$, the rejection event $\{\tilde{\psi}_{n,\delta}=1\}$ is the union $\bigcup_{k\in\bbN}\{\psi^{(k)}_{n,6\delta/(\pi^2k^2)}=1\}$, whose probability (by \Cref{thm:general}) is bounded from the above by $ \sum_{k\in\bbN}6\delta/(\pi^2k^2)=\delta$. The type~I error control is thus proved.

To see the adaptive type~II error control, we first notice that the non-rejection event $\{\tilde{\psi}_{n,\delta}=0\}$ is the intersection $\bigcap_{k\in\bbN}\{\psi^{(k)}_{n,6\delta/(\pi^2k^2)}=0\}$. Thus, for each $k\in\bbN$, we can write
\begin{align*}
    \sup_{(P,Q)\in\cM_{\cS_k}:\ \Toff{P}{Q}\nsucceq \falt} (P^n\times Q^n)(1-\tilde{\psi}_{n,\delta})&\leqslant  \sup_{(P,Q)\in\cM_{\cS_k}:\ \Toff{P}{Q}\nsucceq \falt} (P^n\times Q^n)(1-\psi^{(k)}_{n,6\delta/(\pi^2k^2)})\\
    &\leqslant 6\delta/(\pi^2k^2)\\
    &\leqslant \delta,
\end{align*}
where in the second-to-last line we have applied \Cref{thm:general} to $\cS_k$.

\section{Proofs for \Cref{sec:rates}}\label{app:rates}

\subsection{Proof of \Cref{prop:upper-rate}}\label{sec:proof-separation}

Throughout this section, write $V(t):=t\wedge(1-t)$ for $t\in[0,1]$, and for fixed $\delta\in(0,1)$ set
\[
L_{n,\delta}:=\VC(\cS)\log(2n)+\log(32/\delta),
\qquad
\tau:=\tau_{n,\delta}:=\frac{4L_{n,\delta}}{n}.
\]
Recall $h_+(t):=t+\sqrt{\tau\,V(t)}+\tau$ and $\hinv(t):=\inf\{s\in[0,1]: h_+(s)\geqslant t\}$. Whenever $\fnull$ appears at arguments $>1$, we use the standard extension $\fnull(t):=0$ for $t>1$.

Fix $\alpha\in[0,1]$. If $\fnull(\alpha)=0$, then $\falt(\alpha)=0$ as well since $0\leqslant \falt\leqslant \fnull$, and there is nothing to prove; assume henceforth $\fnull(\alpha)>0$. Set
\[
\Delta(\alpha):=\fnull(\alpha)-\hinvinv\circ \fnull\circ \hupup(\alpha)
=A_\alpha+B_\alpha,
\]
where
\begin{align*}
A_\alpha&:=\fnull(\alpha)-\fnull(\hupup(\alpha)),\\
B_\alpha&:=\fnull(\hupup(\alpha))-\hinvinv\circ \fnull\circ \hupup(\alpha).
\end{align*}

\begin{lemma}\label{lem:helper-lemma}
There exists a universal constant $C_0>0$ such that for all $x,t\in[0,1]$,
\begin{align}
\hupup(x)-x
&\leqslant C_0\bigl(\mgn{x}\bigr),
\label{eq:h-forward-bound}
\\
t-\hinvinv(t)
&\leqslant C_0\bigl(\mgn{t}\bigr).
\label{eq:h-backward-bound}
\end{align}
\end{lemma}

\paragraph{Bound on $A_\alpha$.} By \eqref{eq:h-forward-bound} and monotonicity of $\fnull$,
\begin{equation}
A_\alpha
\leqslant \fnull(\alpha)-\fnull\!\left(\alpha+C_0\bigl(\mgn{\alpha}\bigr)\right).
\label{eq:A-alpha-bound}
\end{equation}

\paragraph{Bound on $B_\alpha$.} Let $y_\alpha:=\fnull(\hupup(\alpha))\in[0,1]$, so $B_\alpha=y_\alpha-\hinvinv(y_\alpha)$. 
Since $\hinv$ is positive and $\fnull$ is monotone, we have $B_\alpha \leqslant \fnull(\alpha)$.

In addition, since $V$ is $1$-Lipschitz and $y_\alpha=\fnull(\alpha)-A_\alpha$, we have $V(y_\alpha)\leqslant V(\fnull(\alpha))+A_\alpha$; combined with \eqref{eq:h-backward-bound} and $\sqrt{a+b}\leqslant\sqrt a+\sqrt b$,
\begin{align*}
B_\alpha
&\leqslant C_0\bigl(\mgn{y_\alpha}\bigr)
\leqslant C_0\Bigl(\tau+\sqrt{\tau\,V(\fnull(\alpha))}+\sqrt{\tau\,A_\alpha}\Bigr).
\end{align*}
Using $2\sqrt{ab}\leqslant a+b$ on $\sqrt{\tau\,A_\alpha}\leqslant \tfrac12 A_\alpha+\tfrac12\tau$,
\begin{equation}
B_\alpha
\leqslant
C_0\left(\tfrac32\tau+\sqrt{\tau\,V(\fnull(\alpha))}+\tfrac12 A_\alpha\right) \wedge \fnull(\alpha).
\label{eq:B-alpha-bound}
\end{equation}

\paragraph{Conclusion.} Combining \eqref{eq:A-alpha-bound}--\eqref{eq:B-alpha-bound} and absorbing constants into a universal $C'$,
\[
\Delta(\alpha)
\leqslant
C'\Bigl[
\fnull(\alpha)-\fnull\!\left(\alpha+C_0\bigl(\mgn{\alpha}\bigr)\right)
+\sqrt{\tau\,V(\fnull(\alpha))} \wedge \fnull (\alpha)+\tau \wedge \fnull (\alpha)
\Bigr].
\]
For simplicity, assume $C_0 > 2$. Note that if $\fnull(\alpha) < \tau$ or $V(\fnull(\alpha)) \geqslant \tau$, then
\[
\tau \wedge \fnull(\alpha)
\leqslant
\sqrt{\tau\,V(\fnull(\alpha))} \wedge \fnull(\alpha).
\]
Otherwise, we have $\fnull(\alpha) > 1 - \tau$. In this case,
\begin{align*}
\fnull(\alpha) - \fnull\!\left(\alpha + C_0 \bigl(\mgn{\alpha}\bigr)\right)
&\geqslant 1 - \tau - \max\{0,\, 1 - (\alpha + C_0 \bigl(\mgn{\alpha}\bigr))\} \\
&\geqslant \min\{1 - \tau,\; (C_0 - 1)\tau\} \\
&\geqslant \tau,
\end{align*}
where the last inequality uses $C_0 > 2$ and $\tau < \tfrac{1}{2}$.

\noindent Therefore,
\[
\tau \wedge \fnull(\alpha)
\leqslant
\fnull(\alpha) - \fnull\!\left(\alpha + C_0 \bigl(\mgn{\alpha}\bigr)\right)
+ \sqrt{\tau\,V(\fnull(\alpha))} \wedge \fnull(\alpha).
\]
And thus we have,
\[
\Delta(\alpha)
\leqslant
2C'\Bigl[
\fnull(\alpha)-\fnull\!\left(\alpha+C_0\bigl(\mgn{\alpha}\bigr)\right)
+\sqrt{\tau\,V(\fnull(\alpha))} \wedge \fnull (\alpha)
\Bigr].
\]
Note that since $\fnull$ is convex, for any $\Delta > 0$ and any $C_1 > 1$, we have
\[
\fnull(\alpha)-\fnull\!\left(\alpha+C_1\Delta\right)\leqslant C_1(\fnull(\alpha)-\fnull\!\left(\alpha+\Delta\right)).
\]
It follows that
\begin{align*}
\fnull(\alpha)-\fnull\!\left(\alpha+C_0\Bigl(\tauND+\sqrt{\tauND\,V(\alpha)}\Bigr)\right)
&\leqslant C_0\left( \fnull(\alpha)-\fnull\!\left(\alpha+\tauND+\sqrt{\tauND\,V(\alpha)}\right)\right),
\end{align*}
Setting $C = 2C'\cdot C_0$, we obtain
\begin{align}\label{eq:gap-upper-rate-proof}
\Delta(\alpha)
\leqslant
C \Bigl[
\fnull(\alpha) - \fnull\!\left(\alpha +\tauND+\sqrt{\tauND\,V(\alpha)}\right)
+ \sqrt{\tauND\, V(\fnull(\alpha))} \wedge \fnull(\alpha)
\Bigr].
\end{align}
Hence \eqref{eq:upper-rate-clean} implies $\falt(\alpha)\leqslant \fnull(\alpha)-\Delta(\alpha)\leqslant \hinvinv\circ \fnull\circ \hupup(\alpha)$, the desired separation condition.

\begin{proof}[Proof of \Cref{lem:helper-lemma}]
Throughout the proof we use repeatedly that $V$ is $1$-Lipschitz on $[0,1]$, that $\sqrt{a+b}\leqslant\sqrt a+\sqrt b$, and that $2\sqrt{ab}\leqslant a+b$.

\paragraph{Forward bound \eqref{eq:h-forward-bound}.} Since $V$ is $1$-Lipschitz,
$V(h_+(x))\leqslant V(x)+(h_+(x)-x)=V(x)+\sqrt{\tau\,V(x)}+\tau$. Hence
\begin{align*}
\hupup(x)-x
&\leqslant 2\tau+\sqrt{\tau\,V(x)}+\sqrt{\tau\,V(h_+(x))}\\
&\leqslant 2\tau+2\sqrt{\tau\,V(x)}+\sqrt{\tau\bigl(\sqrt{\tau\,V(x)}+\tau\bigr)}\\
&\leqslant 2\tau+2\sqrt{\tau\,V(x)}+\bigl(\tau^{3/4}V(x)^{1/4}+\tau\bigr)
\leqslant 4\tau+3\sqrt{\tau\,V(x)},
\end{align*}
where the last step uses $2\sqrt{ab}\leqslant a+b$ with $a=\tau,b=\sqrt{\tau V(x)}$. So \eqref{eq:h-forward-bound} holds with $C_0=4$.

\paragraph{Backward bound \eqref{eq:h-backward-bound}.} Write
$d_1(t):=t-\hinv(t)$ and $d_2(t):=\hinv(t)-\hinvinv(t)$.
Since $h_+(s)\geqslant s$, we have $\hinv(t)\leqslant t$; since $h_+(\hinv(t))\geqslant t$ by definition,
\[
d_1(t)\leqslant h_+(\hinv(t))-\hinv(t)=\tau+\sqrt{\tau\,V(\hinv(t))}\leqslant \tau+\sqrt{\tau\,V(t)}+\sqrt{\tau\,d_1(t)},
\]
using $V(\hinv(t))\leqslant V(t)+d_1(t)$. Applying $\sqrt{\tau\,d_1}\leqslant\tfrac12 d_1+\tfrac12\tau$ and solving,
\begin{equation}
d_1(t)\leqslant 3\tau+2\sqrt{\tau\,V(t)}\leqslant 4\bigl(\mgn{t}\bigr).
\label{eq:d1-bound}
\end{equation}
The same argument with $t$ replaced by $\hinv(t)$ yields $d_2(t)\leqslant 4\bigl(\mgn{\hinv(t)}\bigr)$, and $V(\hinv(t))\leqslant V(t)+d_1(t)$ then gives, via $\sqrt{\tau\,d_1(t)}\leqslant\sqrt{4\tau(\tau+\sqrt{\tau V(t)})}\leqslant 2\tau+2\sqrt{\tau V(t)}$ (by \eqref{eq:d1-bound}),
\[
d_2(t)\leqslant 4\Bigl(\tau+\sqrt{\tau\,V(t)}+\sqrt{\tau\,d_1(t)}\Bigr)\leqslant 12\bigl(\mgn{t}\bigr).
\]
Adding: $t-\hinvinv(t)=d_1(t)+d_2(t)\leqslant 16\bigl(\mgn{t}\bigr)$. So \eqref{eq:h-backward-bound} holds with $C_0=16$.
\end{proof}

\subsection{Proof of \Cref{cor:tvtolerant-vc}}
\label{sec:proof-tvtolerant-vc}
Let \(g_\varepsilon(\alpha):=(1-\varepsilon-\alpha)_+\) for \(\alpha\in[0,1].\)

\noindent For the tolerant total-variation problem, the benchmarks are
\[
\fnull=g_{\varepsilon_1},
\qquad
\falt=g_{\varepsilon_2}.
\]
Recall that
\[
\dTV(P,Q)\leqslant \varepsilon
\quad\Longleftrightarrow\quad
\Toff{P}{Q}\succeq g_\varepsilon .
\]
Thus it suffices to verify the separation condition in
\Cref{thm:general}.

\noindent Write
\[
\tau:=\tau_{n,\delta},
\qquad
m:=\tau+\sqrt{\tau}.
\]
Since \(V(t)\leqslant 1\) for all \(t\in[0,1]\), we have
\[
\hup(t)\leqslant t+m,
\qquad t\in[0,1].
\]
Consequently,
\[
\hup\circ\hup(\alpha)\leqslant \alpha+2m,
\qquad \alpha\in[0,1].
\]
Moreover, from \(\hup(s)\leqslant s+m\), every \(s\) satisfying
\(\hup(s)\geqslant t\) must obey \(s\geqslant t-m\). Hence
\[
\hinv(t)\geqslant (t-m)_+,
\qquad t\in[0,1],
\]
and therefore
\[
\hinv\circ\hinv(t)\geqslant (t-2m)_+.
\]

We now prove the separation condition. If
\(g_{\varepsilon_2}(\alpha)=0\), then the desired inequality is immediate
because the right-hand side is nonnegative. It remains to consider
\(\alpha\) such that \(g_{\varepsilon_2}(\alpha)>0\), namely
\[
\alpha<1-\varepsilon_2.
\]
Assume that
\[
\varepsilon_2-\varepsilon_1\geqslant 4m.
\]
Then, using the monotonicity of \(g_{\varepsilon_1}\), for \(\alpha<1-\varepsilon_2\leqslant 1-\varepsilon_1-4m\)
\begin{align*}
\hinv\circ\hinv\circ g_{\varepsilon_1}
\circ\hup\circ\hup(\alpha)
&\geqslant
\Bigl(g_{\varepsilon_1}(\hup\circ\hup(\alpha))-2m\Bigr)_+ \\
&\geqslant
\Bigl(g_{\varepsilon_1}(\alpha+2m)-2m\Bigr)_+ \\
&=
\bigl(1-\varepsilon_1-\alpha-4m\bigr)_+ \\
&\geqslant 1-\varepsilon_2-\alpha\\
&=g_{\varepsilon_2}(\alpha).
\end{align*}
Thus
\[
g_{\varepsilon_2}(\alpha)
\leqslant
\hinv\circ\hinv\circ g_{\varepsilon_1}
\circ\hup\circ\hup(\alpha),
\qquad \forall \alpha\in[0,1].
\]
This is exactly the separation condition of \Cref{thm:general}.

\noindent Since we assume \(\tau_{n,\delta}\leqslant 1/2\), we have
\[
4m=4(\tau+\sqrt{\tau})\leqslant 8\sqrt{\tau}.
\]
Therefore the separation condition is implied by
\[
\varepsilon_2-\varepsilon_1
\geqslant C\sqrt{\tau_{n,\delta}}
\]
for constant \(C=8\).

Finally, on the alphabet \([k]\), the class \(\cS=2^{[k]}\) contains every
subset of the sample space. Hence every likelihood-ratio level set belongs to
\(\cS\), so every pair of distributions on \([k]\) is \(\cS\)-attainable.
Moreover, \(\VC(\cS)=k\), and therefore the value of \(\tau_{n,\delta}\) is
precisely the one stated in the corollary. Applying
\Cref{thm:general} gives
\[
\errone(\psi_{n,\delta})\leqslant \delta,
\qquad
\errtwo(\cM_{\cS};\psi_{n,\delta})\leqslant \delta.
\]
This proves the corollary.

\subsection{Proof of \Cref{thm:new-lower-bound}}\label{sec:proof-minimax-lower-bound}

The proof of \Cref{thm:new-lower-bound} proceeds by a two-way
reduction.  The first lower bound captures the horizontal uncertainty in
the input coordinate \(\alpha\), including both the interior \(1/\sqrt n\)
scale and the boundary \(1/n\) scale.  The second lower bound captures the
vertical uncertainty in the output coordinate \(\fnull(\alpha)\).

For \(\alpha\in[0,1]\), define
\[
H_n(\alpha)
:=
\fnull(\alpha)
-
\fnull\!\left(
\alpha+\frac1n+\sqrt{\frac{V(\alpha)}{n}}
\right),
\qquad
W_n(\alpha)
:=
\sqrt{\frac{V(\fnull(\alpha))}{n}}\wedge \fnull(\alpha).
\]
With this notation, the local gap in \Cref{thm:new-lower-bound} is
\[
\Delta_{\fnull}(\alpha;1/n)=H_n(\alpha)+W_n(\alpha).
\]

\begin{lemma}[Horizontal local lower bound]\label{prop:lower-horizontal-local-full}
For every \(\delta\in(0,1/2)\), there exists a constant \(c_\delta>0\)
such that if for some \(\alpha\in[0,1]\),
\[
\fnull(\alpha)-\falt(\alpha)
<
c_\delta H_n(\alpha),
\]
then
\[
\inf_{\psi}
\Bigl\{
\errsum{\psi}
\Bigr\}
\geqslant 2\delta.
\]
\end{lemma}

\begin{lemma}[Vertical local lower bound]\label{prop:lower-vertical-local-full}
For every \(\delta\in(0,1/2)\), there exists a constant \(c_\delta>0\)
such that if for some \(\alpha\in[0,1]\),
\[
\fnull(\alpha)-\falt(\alpha)
<
c_\delta W_n(\alpha),
\]
then
\[
\inf_{\psi}
\Bigl\{
\errsum{\psi}
\Bigr\}
\geqslant 2\delta.
\]
\end{lemma}

\begin{proof}[Proof of \Cref{thm:new-lower-bound}]
Fix \(\alpha\in[0,1]\) and write
\[
\Delta_0:=\fnull(\alpha)-\falt(\alpha).
\]
Let
\[
H:=H_n(\alpha),
\qquad
W:=W_n(\alpha).
\]
Let \(c_\delta>0\) be the smaller of the constants in
\Cref{prop:lower-horizontal-local-full} and
\Cref{prop:lower-vertical-local-full}, and set
\[
C_\delta:=\frac{c_\delta}{2}.
\]
Assume that
\[
\Delta_0<C_\delta\{H+W\}
=
\frac{c_\delta}{2}(H+W).
\]
Then either
\[
\Delta_0<c_\delta H
\qquad\text{or}\qquad
\Delta_0<c_\delta W.
\]
In the first case, \Cref{prop:lower-horizontal-local-full} applies.
In the second case, \Cref{prop:lower-vertical-local-full} applies.
Therefore, in either case,
\[
\inf_{\psi}
\Bigl\{
\errsum{\psi}
\Bigr\}
\geqslant 2\delta.
\]
This proves the theorem.
\end{proof}

\subsection{Proofs for the Two Local Lower Bounds}

Define \(P\) to be the uniform distribution on \([0,1]\), with density
\[
p(x)=\mathbf 1_{[0,1]}(x),
\]
and define \(R\) to have density
\[
r(x):=
\mathbf 1_{[\,\fnull(0)-1,\;0\,)}(x)
-\bigl(\fnull\bigr)'_+(x)\,\mathbf 1_{[0,1]}(x).
\]

\begin{lemma}\label{lem:base-pair}
The function \(r\) is a probability density on \(\mathbb R\). Moreover:
\begin{enumerate}
\item \(R((x,\infty))=\fnull(x)\) for every \(x\in[0,1]\);
\item the likelihood ratio \(dR/dP\) is non-increasing on \([0,1]\);
\item \((P,R)\in\cM_{\smlr}\) and \(\Toff{P}{R}=\fnull\).
\end{enumerate}
\end{lemma}

\begin{proof}
Convexity and monotonicity of \(\fnull\) imply that \((\fnull)'_+\) exists
on \((0,1)\), is integrable, and is non-decreasing. Hence
\(-(\fnull)'_+\geqslant0\) on \([0,1]\). Using \(\fnull(1)=0\),
\[
\int_{\mathbb R} r(x)\,dx
=
\bigl(1-\fnull(0)\bigr)
+
\int_0^1 -(\fnull)'_+(x)\,dx
=
\bigl(1-\fnull(0)\bigr)
+
\bigl(\fnull(0)-\fnull(1)\bigr)
=
1.
\]
Thus \(r\) is a probability density.  For every \(x\in[0,1]\),
\[
R((x,\infty))
=
\int_x^1 -(\fnull)'_+(u)\,du
=
\fnull(x),
\]
which proves the first claim.

On \([0,1]\), we have
\[
\frac{dR}{dP}(x)=r(x)=-(\fnull)'_+(x),
\]
which is non-increasing because \((\fnull)'_+\) is non-decreasing.  Thus
the Neyman--Pearson optimal rejection sets are lower tails
\(S_x:=(-\infty,x]\).  Since \(P(S_x)=x\) and
\(R(S_x^\complement)=\fnull(x)\), the trade-off curve is exactly
\(\fnull\).  This proves the second and third claims.
\end{proof}

\subsubsection{Proof of \Cref{prop:lower-horizontal-local-full}}

Let \((P,R)\) be the base pair from \Cref{lem:base-pair}.  Then
\[
(P,R)\in\cM_{\smlr},
\qquad
\Toff{P}{R}=\fnull,
\]
so \((P,R)\) belongs to the null.  Fix \(\alpha\in[0,1]\), and write
\[
s_n(\alpha):=\frac1n+\sqrt{\frac{V(\alpha)}{n}},
\qquad
D_\alpha(t):=\fnull(\alpha)-\fnull(\alpha+t),
\]
where we use the convention \(\fnull(t)=0\) for \(t>1\).  Then
\[
H_n(\alpha)=D_\alpha(s_n(\alpha)).
\]
Since \(\fnull\) is non-increasing and convex, \(D_\alpha\) is
non-decreasing and concave.  If \(H_n(\alpha)=0\), the hypothesis of the
lemma cannot hold, so assume \(H_n(\alpha)>0\).

\paragraph{Interior case.}
First suppose \(V(\alpha)\geqslant 1/n\).  Then \(\alpha\in(0,1)\) and
\[
s_n(\alpha)
=
\frac1n+\sqrt{\frac{V(\alpha)}{n}}
\leqslant
2V(\alpha).
\]
Choose \(\kappa_\delta>0\) such that
\[
\kappa_\delta\leqslant \frac14,
\qquad
4\sqrt 2\,\kappa_\delta\leqslant 1-2\delta,
\]
and set
\[
\theta:=2\kappa_\delta s_n(\alpha).
\]
Then \(\theta\leqslant V(\alpha)\).  Define
\[
p_{\alpha,\theta}(x)
:=
\left(1-\frac{\theta}{\alpha}\right)\mathbf 1_{[0,\alpha]}(x)
+
\left(1+\frac{\theta}{1-\alpha}\right)\mathbf 1_{(\alpha,1]}(x),
\]
and let \(P_{\alpha,\theta}\) be the corresponding probability measure.

Since \(\theta\leqslant V(\alpha)\), \(p_{\alpha,\theta}\geqslant0\) and
\(\int p_{\alpha,\theta}=1\).  Since \(r\) is non-increasing on \([0,1]\)
and \(1/p_{\alpha,\theta}\) jumps downward at \(\alpha\), the likelihood
ratio \(dR/dP_{\alpha,\theta}\) is non-increasing.  Hence
\[
(P_{\alpha,\theta},R)\in\cM_{\smlr},
\]
and lower-tail tests are Neyman--Pearson optimal.

Let \(S_x:=(-\infty,x]\).  The unique \(x_\theta\in(\alpha,1]\) satisfying
\(P_{\alpha,\theta}(S_{x_\theta})=\alpha\) obeys
\[
x_\theta-\alpha
=
\frac{\theta(1-\alpha)}{1-\alpha+\theta}
\geqslant
\frac{\theta}{2}
=
\kappa_\delta s_n(\alpha),
\]
where we used \(\theta\leqslant1-\alpha\).  Therefore
\[
\Toff{P_{\alpha,\theta}}{R}(\alpha)
=
\fnull(x_\theta)
\leqslant
\fnull\bigl(\alpha+\kappa_\delta s_n(\alpha)\bigr).
\]

A direct computation gives
\[
\chi^2(P_{\alpha,\theta}\|P)
=
\theta^2
\left(
\frac1\alpha+\frac1{1-\alpha}
\right)
\leqslant
\frac{2\theta^2}{V(\alpha)}.
\]
Since \(V(\alpha)\geqslant1/n\),
\[
n\frac{s_n(\alpha)^2}{V(\alpha)}
=
n\frac{\left(n^{-1}+\sqrt{V(\alpha)/n}\right)^2}{V(\alpha)}
\leqslant 4.
\]
Thus
\[
\TV(P_{\alpha,\theta}^n,P^n)
\leqslant
\sqrt{n\chi^2(P_{\alpha,\theta}\|P)}
\leqslant
4\sqrt2\,\kappa_\delta
\leqslant
1-2\delta.
\]
By concavity of \(D_\alpha\),
\[
D_\alpha(\kappa_\delta s_n(\alpha))
\geqslant
\kappa_\delta D_\alpha(s_n(\alpha))
=
\kappa_\delta H_n(\alpha).
\]
Therefore, if \(c_\delta\leqslant\kappa_\delta\), the hypothesis
\[
\fnull(\alpha)-\falt(\alpha)<c_\delta H_n(\alpha)
\]
implies
\[
\falt(\alpha)
>
\fnull\bigl(\alpha+\kappa_\delta s_n(\alpha)\bigr).
\]
Hence
\[
\Toff{P_{\alpha,\theta}}{R}(\alpha)<\falt(\alpha),
\]
so \((P_{\alpha,\theta},R)\) belongs to the alternative.  Consequently, for
any test \(\psi\),
\[
\errsum{\psi}
\geqslant
(P^n\times R^n)(\psi)
+
(P_{\alpha,\theta}^n\times R^n)(1-\psi)
\geqslant
1-\TV(P_{\alpha,\theta}^n,P^n)
\geqslant
2\delta.
\]

\paragraph{Boundary case.}
Now suppose \(V(\alpha)<1/n\).  Define
\[
\gamma:=\frac{1-2\delta}{2\sqrt n},
\qquad
\lambda:=\frac{(1-2\delta)^2}{4n},
\]
and
\[
g(t)
=
\max\Bigl\{
0,\,
1-\lambda-e^\gamma t,\,
e^{-\gamma}(1-\lambda-t)
\Bigr\},
\qquad t\in[0,1].
\]
Set
\[
\fstar(t):=\fnull\bigl(1-g(t)\bigr).
\]
Let \(Q\) have density
\[
q(x)
=
e^{-\gamma}\mathbf 1_{[\lambda,\,(e^\gamma+\lambda)/(e^\gamma+1)]}(x)
+
e^\gamma\mathbf 1_{[(e^\gamma+\lambda)/(e^\gamma+1),\,1]}(x)
+
\mathbf 1_{[1,\,1+\lambda]}(x).
\]
Then \(q\geqslant0\) and \(\int q=1\).  Since \(q/p\) is non-decreasing on
the common support and \(dR/dP\) is non-increasing on \([0,1]\), the
likelihood ratio \(dR/dQ\) is non-increasing.  Hence
\[
(Q,R)\in\cM_{\smlr}.
\]
A direct computation of the lower-tail Neyman--Pearson curve gives
\[
\Toff{Q}{R}(t)=\fstar(t),
\qquad t\in[0,1].
\]
The Hellinger affinity between \(P\) and \(Q\) is
\[
\int\sqrt{pq}
=
\frac{1-\lambda}{\cosh(\gamma/2)}.
\]
Thus
\[
H^2(P,Q)
=
1-\frac{1-\lambda}{\cosh(\gamma/2)}
\leqslant
\lambda+\frac{\gamma^2}{2}
\leqslant
\frac{(1-2\delta)^2}{2n}.
\]
By tensorization and \(\TV\leqslant\sqrt2\,H\),
\[
\TV(P^n,Q^n)\leqslant 1-2\delta.
\]
Let
\[
t_{\mathrm{bd}}:=\frac1n\wedge(1-\alpha).
\]
From the definition of \(g\),
\[
1-\alpha-g(\alpha)
=
\min\Bigl\{
1-\alpha,\,
\lambda+(e^\gamma-1)\alpha,\,
(1-e^{-\gamma})(1-\alpha)+e^{-\gamma}\lambda
\Bigr\}.
\]
Since \(\lambda\asymp_\delta n^{-1}\), there exists a constant
\(\rho_\delta>0\) such that
\[
1-\alpha-g(\alpha)\geqslant \rho_\delta t_{\mathrm{bd}}
\qquad
\text{for all }\alpha\in[0,1].
\]
Therefore
\[
\fnull(\alpha)-\fstar(\alpha)
=
D_\alpha\bigl(1-\alpha-g(\alpha)\bigr)
\geqslant
D_\alpha(\rho_\delta t_{\mathrm{bd}})
\geqslant
\rho_\delta D_\alpha(t_{\mathrm{bd}}).
\]

Since \(V(\alpha)<1/n\), if \(\alpha\leqslant1/2\), then
\(t_{\mathrm{bd}}=1/n\) and
\[
s_n(\alpha)
=
\frac1n+\sqrt{\frac{V(\alpha)}{n}}
\leqslant
\frac2n
=
2t_{\mathrm{bd}}.
\]
Hence
\[
H_n(\alpha)
=
D_\alpha(s_n(\alpha))
\leqslant
D_\alpha(2t_{\mathrm{bd}})
\leqslant
2D_\alpha(t_{\mathrm{bd}}).
\]
If \(\alpha>1/2\), then \(t_{\mathrm{bd}}=1-\alpha<1/n\), so
\[
\alpha+s_n(\alpha)>1,
\]
and therefore
\[
H_n(\alpha)=\fnull(\alpha)=D_\alpha(t_{\mathrm{bd}}).
\]
In either case,
\[
\fnull(\alpha)-\fstar(\alpha)
\geqslant
\frac{\rho_\delta}{2}H_n(\alpha).
\]
Choosing \(c_\delta\leqslant\rho_\delta/2\), the hypothesis
\[
\fnull(\alpha)-\falt(\alpha)<c_\delta H_n(\alpha)
\]
implies
\[
\falt(\alpha)>\fstar(\alpha).
\]
Therefore \((Q,R)\) belongs to the alternative.  Hence, for any test
\(\psi\),
\[
\errsum{\psi}
\geqslant
(P^n\times R^n)(\psi)
+
(Q^n\times R^n)(1-\psi)
\geqslant
1-\TV(P^n,Q^n)
\geqslant
2\delta.
\]
Combining the interior and boundary cases proves
\Cref{prop:lower-horizontal-local-full}.

\subsubsection{Proof of \Cref{prop:lower-vertical-local-full}}
Let \((P,R)\) be the base pair from \Cref{lem:base-pair}, and set
\[
a:=\fnull(\alpha)=R((\alpha,\infty)).
\]
If \(a=0\), then \(W_n(\alpha)=0\), and the hypothesis of the lemma cannot
hold.  Hence assume \(a>0\).

Choose \(\kappa_\delta>0\) such that
\[
\sqrt2\,\kappa_\delta\leqslant 1-2\delta,
\]
and set
\[
\theta:=\kappa_\delta W_n(\alpha).
\]
Since \(W_n(\alpha)\leqslant a\), we have \(\theta\leqslant a\).  Let
\[
A_\alpha:=(-\infty,\alpha],
\qquad
B_\alpha:=(\alpha,\infty),
\]
and define
\[
r_{\alpha,\theta}(x)
:=
\left(1+\frac{\theta}{1-a}\right)r(x)\mathbf 1_{A_\alpha}(x)
+
\left(1-\frac{\theta}{a}\right)r(x)\mathbf 1_{B_\alpha}(x).
\]
Let \(R_{\alpha,\theta}\) be the corresponding probability measure.  The
density is non-negative and integrates to one because
\[
R(A_\alpha)=1-a,
\qquad
R(B_\alpha)=a.
\]
Since \(r\) is non-increasing and the multiplicative factor in
\(r_{\alpha,\theta}\) jumps downward at \(\alpha\), the likelihood ratio
\(dR_{\alpha,\theta}/dP\) is non-increasing.  Therefore
\[
(P,R_{\alpha,\theta})\in\cM_{\smlr},
\]
and lower-tail tests are Neyman--Pearson optimal.  Since
\(P(A_\alpha)=\alpha\), the test \(A_\alpha\) has level \(\alpha\), and
\[
\Toff{P}{R_{\alpha,\theta}}(\alpha)
=
R_{\alpha,\theta}(B_\alpha)
=
\left(1-\frac{\theta}{a}\right)a
=
\fnull(\alpha)-\theta.
\]
A direct computation gives
\[
\chi^2(R_{\alpha,\theta}\|R)
=
(1-a)\left(\frac{\theta}{1-a}\right)^2
+
a\left(\frac{\theta}{a}\right)^2
=
\theta^2\left(\frac1{1-a}+\frac1a\right)
\leqslant
\frac{2\theta^2}{V(a)}.
\]
Since \(W_n(\alpha)^2\leqslant V(a)/n\),
\[
\TV(R_{\alpha,\theta}^n,R^n)
\leqslant
\sqrt{n\,\chi^2(R_{\alpha,\theta}\|R)}
\leqslant
\sqrt2\,\kappa_\delta
\leqslant
1-2\delta.
\]
Choose the constant \(c_\delta\) in the statement of the lemma so that
\(c_\delta\leqslant\kappa_\delta\).  If
\[
\fnull(\alpha)-\falt(\alpha)<c_\delta W_n(\alpha),
\]
then
\[
\falt(\alpha)
>
\fnull(\alpha)-\kappa_\delta W_n(\alpha)
=
\fnull(\alpha)-\theta
=
\Toff{P}{R_{\alpha,\theta}}(\alpha).
\]
Thus \((P,R_{\alpha,\theta})\) belongs to the alternative, while
\((P,R)\) belongs to the null.  Therefore, for any test \(\psi\),
\[
\errsum{\psi}
\geqslant
(P^n\times R^n)(\psi)
+
(P^n\times R_{\alpha,\theta}^n)(1-\psi)
\geqslant
1-\TV(R_{\alpha,\theta}^n,R^n)
\geqslant
2\delta.
\]
This proves \Cref{prop:lower-vertical-local-full}.

\section{Proofs for \Cref{sec:misspecification}}\label{app:misspecification}
\subsection{Proof of \Cref{thm:misspecified}}\label{sec:proof-misspecified}
Type~I error control is a straightforward corollary of \Cref{thm:general}, so it suffices to establish type~II. By \Cref{lem:nuc}, with probability at least $1-\delta$, the event
\[
\mathcal E=\Bigl\{\forall S\in\cS\cup\cS^\complement:\ P(S)\leqslant\hup(P_n(S)),\ P_n(S)\leqslant\hup(P(S)),\ Q(S)\leqslant\hup(Q_n(S)),\ Q_n(S)\leqslant\hup(Q(S))\Bigr\}
\]
holds; we work on $\mathcal E$.

\paragraph{Type II error control.}
Assume $\Toff{P}{Q}\nsucceq\falt$, which implies the existence of some $\alpha \in [0,1]$ such that $0 \leqslant \Toff{P}{Q}(\alpha) < \falt(\alpha)$. 
Under the gap condition, we further have 
\begin{align*}
0 \leqslant \Toff{P}{Q}(\alpha)
&<\max \{ \hinv\circ\widetilde{\fnull}\circ\hup(\alpha+2\eta_1)-2\eta_2,0 \},
\end{align*}
Clearly this requires 
$\hinv\circ\widetilde{\fnull}\circ\hup(\alpha+2\eta_1)-2\eta_2 > 0$. As a result, we obtain 
\begin{align*}
0 \leqslant \Toff{P}{Q}(\alpha)
&< \hinv\circ\widetilde{\fnull}\circ\hup(\alpha+2\eta_1)-2\eta_2.
\end{align*}
Now since $(P,Q)$ is $(\cS,\eta_1,\eta_2)$-attainable, we know that for any $\alpha \in [0,1]$
\[
\Toff{P}{Q}(\alpha)\geqslant \Toff{P'}{Q'}(\alpha+\eta_1)-\eta_2.
\]
Consequently, combining the above two displays, we have 
\begin{align*}
\Toff{P'}{Q'}(\alpha + \eta_1)
&<\hinv\circ\widetilde{\fnull}\circ\hup(\alpha+2\eta_1)-\eta_2.
\end{align*}
By \Cref{lem:witness}, there exists $S^\star\in\cS$ with $0\leqslant Q'((S^\star)^\complement)<\hinv\circ\widetilde{\fnull}\circ\hup\bigl(P'(S^\star)+\eta_1\bigr)-\eta_2$. Since $\hup$ and $\hinv$ are non-decreasing and $\fnull$ is non-increasing,  $\hinv\circ\widetilde{\fnull}\circ\hup$ is non-increasing. Combining $\dTV(P,P')\leqslant\eta_1$, $\dTV(Q,Q')\leqslant\eta_2$, we have
\[
Q((S^\star)^\complement)
\leqslant Q'((S^\star)^\complement)+\eta_2
<\hinv\circ\widetilde{\fnull}\circ\hup\bigl(P'(S^\star)+\eta_1\bigr)
\leqslant \hinv\circ\widetilde{\fnull}\circ\hup\bigl(P(S^\star)\bigr)\leqslant \hinv\circ\widetilde{\fnull}(P_n(S^\star)).
\]
By \Cref{lem:hinv-properties}(iii), this is equivalent to $\hup(Q((S^\star)^\complement))<\widetilde{\fnull}(P_n(S^\star))$. On $\mathcal E$ we have $Q_n((S^\star)^\complement)\leqslant\hup(Q((S^\star)^\complement))$, so $Q_n((S^\star)^\complement)<\widetilde{\fnull}(P_n(S^\star))$. \Cref{prop:equiv-surrogate-convex} then shows that the test~\eqref{eq:general-test} rejects on $\mathcal E$, giving $\errtwo(\cM_{\cS, \eta_1,\eta_2};\psi_{n,\delta})\leqslant\delta$.

\subsection{Proof of \Cref{lem:log-concave-approx-attain}}\label{sec:log-concave}

We rely on two lemmas: a standard TV-approximation result and an attainability fact for piecewise-linear densities.

\begin{lemma}[Lemma 27 of \cite{chan2014efficient}, setting \(t = \lceil\frac{1}{2\sqrt{\eta}}\rceil\))] \label{TV-approx}
    For any log-concave density $P$ on $\bbR$, any $\eta>0$, there exists a piecewise linear density $Q$ with at most $\lceil\frac{1}{2\sqrt{\eta}}\rceil$ pieces, such that $\TV(P,Q)\leqslant \widetilde{O}(\eta)$.
\end{lemma}

\begin{lemma} \label{lem:piecewise-attainable}
    Any pair of piecewise linear distributions on $\bbR$ with at most $t$-pieces is $\cI_{2t-1}$-attainable.
\end{lemma}

\begin{proof}
Let $f,g$ be piecewise linear densities each with at most $t$ pieces, and for $\lambda\geqslant 0$ set $h_\lambda(x):=f(x)-\lambda g(x)$; the level set of interest is $\{x:f(x)>\lambda g(x)\}=\{h_\lambda>0\}$. The $t$ linear pieces of $f$ are delimited by $t+1$ pivots (including $\pm\infty$), and similarly for $g$; merging gives at most $2t+2$ pivots, between consecutive ones both $f$ and $g$ are linear and so is $h_\lambda$, so $h_\lambda$ is piecewise linear with at most $2t-1$ pieces. On each such piece $\{h_\lambda>0\}$ contributes at most one interval; the conclusion follows.
\end{proof}

The above two technical lemmas imply that
\[
\cM_{\mathrm{LC}}
=
\bigl\{(P,Q):\; P,Q \text{ are log-concave on } \bbR\bigr\}
\subseteq
\cM_{\cI_{\lceil 1/\sqrt{\eta}\rceil},
\widetilde{O}(\eta),\widetilde{O}(\eta)} .
\]


\subsection{Proof of \Cref{cor:log-concave-tolerant}}\label{sec:proof-cor-log-concave}
Take
\[
\eta := n^{-2/5},
\qquad
K := \lceil\frac{1}{\sqrt{\eta}}\rceil=\lceil n^{1/5}\rceil,
\qquad
\cS:=\cI_K .
\]
By \Cref{lem:log-concave-approx-attain}, we have
\[
\cM_{\mathrm{LC}} \subseteq \cM_{\cS,\widetilde{\eta},\widetilde{\eta}}.
\]
for some \(\widetilde{\eta}=\widetilde{O}({\eta})\).

\noindent Since \(\VC(\cI_K)=2K\), it follows that
\[
\tauND
=
\frac{4(\VC(\cS)\log(2n)+\log(32/\delta))}{n}
=
O\!\left(n^{-4/5}\log n\right).
\]
The type~I error bound \(\errone(\psi_{n,\delta})\leqslant \delta\) follows directly from \Cref{thm:misspecified}(i).

\noindent For the type~II error, write
\[
\fnull=g_{\varepsilon_1},
\qquad
\falt=g_{\varepsilon_2},
\qquad
g_\varepsilon(\alpha)=(1-\varepsilon-\alpha)_+.
\]
By \Cref{thm:misspecified}(ii), it suffices to verify the separation condition
\[
\forall\, \alpha \in [0,1]:\qquad g_{\varepsilon_2}(\alpha) \;\leqslant\; \max\big\{\hinvinv\circ g_{\varepsilon_1}\circ \hupup(\alpha + 2\widetilde{\eta})-2\widetilde{\eta},\;0\big\}.
\]
That is,
\begin{align}\label{eq:tv-log-concave-sufficient-1}
    \forall\, \alpha \in [0,1]:\qquad g_{\varepsilon_2}(\alpha) \;\leqslant\; \max\big\{g_{\varepsilon_1}(\alpha+2\widetilde{\eta})-\Delta(\alpha+2\widetilde{\eta})-2\widetilde{\eta},\;0\big\},
\end{align}
where
\begin{align}\label{eq:tv-log-concave-delta}
\Delta(\alpha)&:=g_{\varepsilon_1}(\alpha)-\hinvinv\circ g_{\varepsilon_1}\circ \hupup(\alpha)\nonumber\\
&\leqslant
C \Bigl[
g_{\varepsilon_1}(\alpha) - g_{\varepsilon_1}\!\left(\alpha + C \bigl(\mgn{\alpha}\bigr)\right)
+ \sqrt{\tau\,V(g_{\varepsilon_1}(\alpha))} \wedge g_{\varepsilon_1}(\alpha)
\Bigr],
\end{align}
by \eqref{eq:gap-upper-rate-proof}.

\noindent For the first term on the right-hand side of \eqref{eq:tv-log-concave-delta}, note that \(g_{\varepsilon_1}(\alpha)=(1-\varepsilon_1-\alpha)_+\) is non-increasing and \(1\)-Lipschitz. Hence,
\[
g_{\varepsilon_1}(\alpha)-g_{\varepsilon_1}\!\left(\alpha+C\bigl(\tauND+\sqrt{\tauND\,V(\alpha)}\bigr)\right)
\leqslant
C\bigl(\tauND+\sqrt{\tauND\,V(\alpha)}\bigr).
\]
Since \(V(\alpha)\leqslant 1\), this implies
\[
g_{\varepsilon_1}(\alpha)-g_{\varepsilon_1}\!\left(\alpha+C\bigl(\tauND+\sqrt{\tauND\,V(\alpha)}\bigr)\right)
\lesssim
\bigl(\tauND+\sqrt{\tauND}\bigr).
\]
For the second term on the right-hand side of \eqref{eq:tv-log-concave-delta}, using that \(V(g_{\varepsilon_1}(\alpha))\leqslant 1\), we have
\[
\Bigl(\sqrt{\tauND\,V(g_{\varepsilon_1}(\alpha))}\Bigr)\wedge g_{\varepsilon_1}(\alpha)
\leqslant
\sqrt{\tauND}.
\]
Therefore, for any \(\alpha\in [0,1]\),
\[
\Delta(\alpha+2\widetilde{\eta})\lesssim \tauND+\sqrt{\tauND} \asymp n^{-2/5}.
\]
Combining this with \(g_{\varepsilon_1}(\alpha+2\widetilde{\eta})\geqslant g_{\varepsilon_1}(\alpha)-2\widetilde{\eta}\) and \(\widetilde{\eta}\asymp n^{-2/5}\), a sufficient condition for \eqref{eq:tv-log-concave-sufficient-1} is
\begin{align*}
    \forall\, \alpha \in [0,1]:\qquad g_{\varepsilon_2}(\alpha) \;\leqslant\; \max\big\{g_{\varepsilon_1}(\alpha)-\widetilde{O}(n^{-2/5}),\;0\big\}.
\end{align*}
This condition is satisfied whenever \(\varepsilon_2-\varepsilon_1\geqslant \widetilde{O}(n^{-2/5})\).

\noindent This completes the proof.

\section{Proof of~\Cref{prop:computation-main}}\label{sec:proof-computation}

\subsection{Step 1: a surrogate benchmark}\label{sec:surrogate}

At the population level, the null is violated whenever there exists a set
\(S\in\cS\) such that
\[
Q(S^\complement)<\fnull(P(S)).
\]
The empirical test in~\eqref{eq:test-comp-start} has the same form, but with
confidence corrections applied to both arguments. The first step is to absorb
those corrections into a modified benchmark curve.

The next proposition shows that the rejection rule can be written in the same
geometric form as the population benchmark test, but with \(\fnull\) replaced
by \(\widetilde \fnull\).

\begin{proposition}\label{prop:equiv-surrogate-convex}
The test~\eqref{eq:test-comp-start} rejects if and only if
\begin{equation}\label{eq:test-surrogate}
\exists\, S \in \cS:\quad
Q_n(S^\complement) \;<\; \widetilde{f}_0\bigl(P_n(S)\bigr).
\end{equation}
Moreover, \(\widetilde{\fnull}\) is convex and non-increasing.
\end{proposition}

\noindent See \Cref{lem:hinv-properties}(iii) for the proof of the first
claim, and \Cref{lem:tilde-f} for the proof of the second.
\medskip

This reformulation is conceptually useful: instead of comparing
confidence-corrected empirical quantities, we compare the raw empirical pair
\[
\bigl(P_n(S),\,Q_n(S^\complement)\bigr)
\]
to a modified benchmark curve \(\widetilde \fnull\) that already incorporates
the finite-sample correction. For the algorithmic reduction below, only the
convexity of \(\widetilde \fnull\) matters.

\subsection{Step 2: reduction to a piecewise-linear boundary}\label{sec:discretization}

Since \(P_n(S) \in \{0, 1/n, \ldots, 1\}\), the test~\eqref{eq:test-surrogate}
evaluates \(\widetilde{f}_0\) only on the \(n+1\) grid points
\(\{k/n\}_{k=0}^n\). Let \(g\) denote the piecewise-linear interpolation of
\(\widetilde{f}_0\) with these grid points as breakpoints, that is,
\begin{align}\label{eq:piecewise_linear_interpolation}
g(t)=\begin{dcases}
\widetilde{\fnull}(t),\ &t=0,\frac{1}{n},\frac{2}{n},\cdots,1,\\
\widetilde{\fnull}\left(\frac{k-1}{n}\right)
+n\left[\widetilde{\fnull}\left(\frac{k}{n}\right)-\widetilde{\fnull}\left(\frac{k-1}{n}\right)\right]\left(t-\frac{k-1}{n}\right),
\ &\frac{k-1}{n}<t< \frac{k}{n}.
\end{dcases}
\end{align}
By \Cref{prop:equiv-surrogate-convex}, \(\widetilde{f}_0\) is
convex, so the discrete second differences are nonnegative, which implies that
\(g\) is also convex. Since the test only queries grid points, we obtain the
following equivalence.

\begin{proposition}\label{prop:reduction_to_piecewise_linear}
The test~\eqref{eq:test-surrogate} is equivalent to
\begin{equation}\label{eq:test-pwl}
\exists\, S \in \cS:\quad
Q_n(S^\complement) \;<\; g\bigl(P_n(S)\bigr).
\end{equation}
\end{proposition}

Thus the computational problem reduces to checking whether the empirical point
corresponding to some set \(S\) falls below a convex piecewise-linear boundary.

\subsection{Step 3: reduction to weighted ERM}\label{sec:reduction-cover}

The final simplification comes from the convexity of the piecewise-linear
function \(g\). For each \(k=1,2,\cdots,n\), let
\[
g_k(t)=c_k+\lambda_k t
\]
be the linear piece of \(g\) on \([(k-1)/n,k/n]\).

Since \(g\) is convex, we have
\[
g(t)=\max_{k\in[n]} g_k(t).
\]
Therefore,
\[
Q_n(S^\complement)<g(P_n(S))
\]
holds if and only if there exists some \(k\in\{1,\dots,n\}\) such that
\[
Q_n(S^\complement)-\lambda_k P_n(S)<c_k.
\]
This yields the main computational reduction.

\end{document}